\documentclass[reqno,12pt]{amsart}

\usepackage{amsmath,amsfonts,amssymb,amsthm,epsfig,mathrsfs}

\usepackage{cleveref}

\title{Asymptotic behaviour of fractional seminorms}
\date{12 June 2025 (last modified)}

\usepackage[colorinlistoftodos,prependcaption,textsize=tiny]{todonotes}

plus 6pt minus 12pt
\theoremstyle{plain}
\newtheorem{thm}{Theorem}[section]
\newtheorem{lemma}[thm]{Lemma}

\newtheorem{cor}[thm]{Corollary}
\newtheorem{prop}[thm]{Proposition}
\newtheorem{ques}[thm]{Question}

\theoremstyle{definition}
\newtheorem{defn}[thm]{Definition}
\newtheorem{rmk}[thm]{Remark}
\numberwithin{equation}{section}

\newcommand{\R}{\mathbb{R}}

\newcommand{\SSS}{\mathscr{S}}

\def\Xint#1{\mathchoice
   {\XXint\displaystyle\textstyle{#1}}%
   {\XXint\textstyle\scriptstyle{#1}}%
   {\XXint\scriptstyle\scriptscriptstyle{#1}}%
   {\XXint\scriptscriptstyle\scriptscriptstyle{#1}}%
   \!\int}
\def\XXint#1#2#3{{\setbox0=\hbox{$#1{#2#3}{\int}$}
     \vcenter{\hbox{$#2#3$}}\kern-.5\wd0}}

\def\dashint{\Xint-}

\author{Ahmed Dughayshim}
\address{Department of Mathematics,
University of Pittsburgh,
417 Thackeray Hall,
Pittsburgh, PA 15260, USA.}
\email{aha80@pitt.edu}


\begin{document}

\begin{abstract}
We obtain sharp identifications of fractional Sobolev spaces $ W^{s}_{p,q}$, extension spaces $E^{s}_{p,q}$, and Triebel-Lizorkin spaces $\dot{F}^s_{p,q}$. As a consequence of these sharp identifications, we obtain for $W^{s}_{p,q}$ and $E^{s}_{p,q}$ a stability theory a la Bourgain-Brezis-Mironescu as $s \to 1^{-}$, answering a question raised by Brazke--Schikorra--Yung. Part of the results are new even for $p=q$.
\end{abstract}
\maketitle

\section{Introduction}
To begin, define the Gagliardo semi-norm for $s \in (0,1)$, $\Omega \subset \R^{n}$ open 
$$
[f]_{W^{s}_{p}(\Omega)} = \left(\int_{\Omega} \int_{\Omega} \frac{ \vert f(x) - f(z) \vert^{p}}{\vert x- z \vert^{n+sp}} dz  dx\right)^{1/p}.
$$
If $ s = 1 $,  then one can show that the only smooth functions with the property $[ f]_{W^{s}_{p}(\Omega)} < \infty$ are locally constant functions on $\Omega$. However, 
in the celebrated work of Bourgain, Brezis and Mironescu \cite{BBM} they prove, among other things, the following two results 
\begin{thm}[BBM1]\label{BBM1}
Let $ p \in (1,\infty)$ and $ \Omega $ be a smooth bounded domain. Then there exists a constant $ C = C(n,p)$ so that 
$$
\lim_{s \to 1^{-}} (1-s)^{1/p} [f]_{W^{s}_{p}(\Omega)} =C \Vert \nabla f \Vert_{L^{p}(\Omega)}
$$
for all $ f \in W^{1,p}(\Omega)$.
\end{thm}

\begin{thm}[BBM2]\label{BBM2}
Let $ p \in (1,\infty)$, and $ \Omega $ be a bounded domain in $ \mathbb{R}^{n}$ with smooth boundary. Let  $ f_{k} \in C^{\infty}_{0}(\Omega)$. Assume $f_{k} \to f $ weakly in $L^{p}(\Omega)$, and 
$$
\Lambda= \sup_{k} ( \Vert f_{k} \Vert_{L^{p}(\Omega)} +  (1-s_{k})^{1/p} [f_{k}]_{W^{s_{k}}_{p}(\Omega)} ) < \infty.
$$
Where $ s_{k} \in (0,1)$ is some sequence converging to $ 1 $ from below. Then $ f \in W^{1,p}(\Omega)$ and we have that $ \Vert f \Vert_{W^{1,p}(\Omega)} \leq C \Lambda.$
\end{thm}
In \cite{P}, Brazke, Schikorra, and Yung give a different proof of Theorem \ref{BBM2} above for when $ \Omega = \mathbb{R}^{n}$. Their approach is via the homogeneous Triebel-Lizorkin spaces $ \dot{F}^{s}_{p,q}$. Namely, let $W^{s}_{p}:= W^{s}_{p}(\R^{n}) $ , then they prove inequalities of the form 
\begin{equation}\label{AP1}
[f]_{W^{s}_{p}} \lesssim_{n,p} (\min \{s, (1-s) \})^{-1/\sigma}[f]_{\dot{F}^{s}_{p,p}},
\end{equation}
where $ \sigma = \min \{ 2,p \}$ and $ s \in (0,1)$. And from below they prove for $ p \in (1,2]$
\begin{equation}\label{AP2}
(\min \{ s , (1-s)\})^{-1/p} [f]_{\dot{F}^{s}_{p,2}} \lesssim_{n,p} [f]_{W^{s}_{p}}.
\end{equation}
They used these inequalities, and other variants of them, to prove a global version of Theorem \ref{BBM2}. Then they posed the question of whether the estimates they obtained would still hold true for $f \in W^{s}_{p,q}$. Here $ W^{s}_{p,q}$ is the natural generalization of $ W^{s}_{p}$. That is,
$$
[f]^{p}_{W^{s}_{p,q}} =\int_{\mathbb{R}^{n}} \left( \int_{\mathbb{R}^{n}} \frac{\vert f(x+z) - f(x) \vert^{q}}{\vert z \vert^{n+sq}} dz \right)^{p/q} dx.
$$
In that regard, Mohanta, in \cite{Mo}, extends some of the results obtained in \cite{BBM} for $ p \neq q $. Namely, he proves \Cref{BBM1} for $ W^{s}_{p,q}$ for a certain range of $ q $. In light of the results obtained in \cite{Mo} it is natural to guess that identifications similar to \eqref{AP1}, and \eqref{AP2} above would still hold true for $W^{s}_{p,q}$ when $ p \neq q$.
 We show that this is indeed the case when $ p$ is not too "far" from $ q $. Our argument is similar in spirit to the one introduced in \cite{P}. In fact, we show that a variant of their argument implies various sharp estimates for another natural fractional spaces, based on the harmonic extension which  for fractional PDE was popularized by Caffarelli-Silvestre \cite{Caf}.
\subsection{Harmonic extension seminorm} For $ p,q \in [1,\infty)$ and $ s \in (0,1)$ define the following seminorm 
$$
[f]_{E^{s}_{p,q}} = \left( \int_{\mathbb{R}^{n}} \left( \int_{0}^{\infty} t^{q(1-s)-1} \vert \partial_{t} P_{t} * f(x) \vert^{q} dt \right)^{p/q} dx \right)^{1/p},
$$
where $P_{t}$ is the Poisson kernel, see definition \ref{defpoission}. The function $ F(x,t) = P_{t}*f(x)$ is sometimes called the harmonic extension of $f$
\[
\begin{cases}
\Delta_{x,t} F(x,t) = 0 \quad &\text{in $\R^n \times (0,\infty)$}\\
F(x,0) = f(x). \quad 
\end{cases}
\]
It is known that $ [f]_{E^{s}_{p,q}} \approx_{s,n,p,q} [f]_{\dot{F}^{s}_{p,q}}$, for an overview see \cite{C} and \cite{HT}, see also \cite{wang} for other similar identifications of the Triebel-Lizorkin space. Similarly if $ p > nq/(n+sq)$ then $ [f]_{\dot{F}^{s}_{p,q}} \approx_{s,n,p,q} [f]_{W^{s}_{p,q}}$, see \cite{Prats} for a proof. In both of the equivalence the dependency on $ s $ is not known. And it is one of the goals of this paper to understand the dependency on $ s $ (at least when $ s  $ is not near $0$) and establish sharp inequalities between these spaces. We do so by first relating $ E_{p,q}^{s} $ to $ \dot{F}^{s}_{p,q}$. This seems to be new, even for $p=q$. In what follows, $ \mathscr{S}(\R^{n})$ denotes the class of Schwartz functions.

\begin{thm}\label{q12}
Let $ f \in \mathscr{S}(\mathbb{R}^{n}) $ and $ p,q \in (1,\infty)$, and $ s \in(0,1)$ then one has 
\begin{enumerate}
    \item if $ q\in (1,2]$ then 
    $$
   {(1-s)^{-1/2}} [f]_{\dot{F}^{s}_{p,q}} \lesssim_{n,p,q} [f]_{E_{p,q}^{s}} \lesssim_{n,p,q} {(1-s)^{-1/q}} [f]_{\dot{F}^{s}_{p,q}}.
    $$
    \item If $ q \in [2, \infty)$ then we have 
    $$ 
{(1-s)^{-1/q}} [f]_{\dot{F}^{s}_{p,q}} \lesssim_{n,p,q} [f]_{E_{p,q}^{s}} \lesssim_{n,p,q} {(1-s)^{-1/2}}[f]_{\dot{F}^{s}_{p,q}}.
     $$
\end{enumerate}
\end{thm} 
As in \cite{P} we see that as we increase the secondary parameter of the norm, the dependency on $s$ improves.

\begin{thm}\label{QA1} Let $ p,q \in (1,\infty)$ and $ f \in \SSS(\R^{n})$, then 
\begin{enumerate}
    \item  if $ q \in (1,2]$ then we have
$$
{(1-s)^{-1/q}} [f]_{\dot{F}^{s}_{p,2}} \lesssim_{n,p,q} [f]_{E^{s}_{p,q}}.
$$
 \item  If $ q \in [2, \infty) $ then we have 
     $$
     [f]_{E_{p,q}^{s}}  \lesssim_{n,p,q} {(1-s)^{-1/q}}[f]_{\dot{F}^{s}_{p,2}}.
     $$     
\end{enumerate}
\end{thm}
From the above it is easy to deduce a version of Theorem \ref{BBM2}  for $ E^{s}_{p,q}$ when $ q \in (1,2]$. However, the general result for $ q \in (1,\infty)$ follows from the Sobolev estimates below.
\begin{thm}\label{SUU}
Let $ s \in (0,1)$ and assume $ \sigma , v \in [0,1]$ with 
$$
v < s < \sigma.
$$
Then we have for any $ f \in \SSS(\R^{n})$ the following
$$
[f]_{E^{s}_{p,q}} \lesssim_{n,p,q} (\sigma-s)^{-1/q} [f]_{\dot{F}^{\sigma}_{p,2}} + (s-v)^{-1/q} [f]_{\dot{F}^{v}_{p,2}}.
$$
\end{thm}
And similar estimates from below hold
\begin{thm}\label{LowerS}
Let $ \Theta > 1 $ and $ s \in (1-\frac{1}{2\Theta}, 1)$. Then for $ \sigma \in (0, \bar{\sigma})$, where $ \bar{\sigma}$ is taken so that $ (1- \bar{\sigma}) = \Theta (1-s)$,  we have for any $ f \in \SSS(\R^{n})$ the following
$$
[f]_{\dot{F}^{\sigma}_{p,2}} \leq C(\Theta, n, p ,q) \left( \Vert f \Vert_{L^{p}} + {(1-s)^{1/q}} [f]_{E^{s}_{p,q}} \right).
$$
\end{thm}
As a consequence of the estimates above we obtain the following BBM theorems.
\begin{thm}\label{BBM1PR}
Let $ p,q \in (1,\infty)$. Then 
$$
\lim_{s \to 1^{-}}(1-s)^{1/q} [f]_{E^{s}_{p,q}} \approx_{n,p,q} \Vert \nabla f \Vert_{L^{p}}
$$
for all $ f\in \SSS(\R^{n}).$
\end{thm}
Also, we have BBM2
\begin{cor}\label{BBM2PR}
Let $ p,q \in (1,\infty)$. Let  $ f_{k} \in \SSS (\mathbb{R}^{n})$, and assume $f_{k} \to f $ weakly in $L^{p}(\mathbb{R}^{n})$ with 
$$
\Lambda= \sup_{k} ( \Vert f_{k} \Vert_{L^{p}} +  (1-s_{k})^{1/q} [f_{k}]_{E^{s}_{p,q}} ) < \infty,
$$
where $ s_{k} \in (0,1)$ is some sequence converging to $ 1 $ from below. Then $ f \in W^{1,p}(\mathbb{R}^{n})$ and we have that $ \Vert f \Vert_{W^{1,p}} \lesssim_{n,p,q}  \Lambda.$
\end{cor}

\subsection{ Estimates for $W^{s}_{p,q}$}
Everything we have above is true for $ W^{s}_{p,q}$ if any of the following conditions hold true 
\begin{enumerate}
    \item $ q < \frac{n}{s} $, and $ p \in (\frac{nq}{n+sq}, \frac{nq}{n-sq})$
    \item $ q \geq n/s$ and $ p \in (nq/(n+sq) , \infty)$.
\end{enumerate}
The main difference is that, unlike $E^{s}_{p,q}$, $W^{s}_{p,q}$ is not stable as $s \to 0$. Therefore, in all of the results we will require that $s$ stays away from $0$. 
More precisely we have 
\begin{thm}\label{1.9}
Let $p,q \in (1,\infty)$ and $ p > \frac{nq}{n+\theta q}$ for some $ \theta \in (0,1)$. And $ p' > \frac{nq'}{n+\theta q'}$. Then for $ s \in (\theta,1) $ we have
\begin{enumerate}
    \item If $ q \in (1,2]$ then  
    $$
   (1-s)^{-1/2}[f]_{\dot{F}^{s}_{p,q}} \lesssim_{n,p,q,\theta} [f]_{W^{s}_{p,q}} \lesssim_{n,p,q,\theta} (1-s)^{-1/q}[f]_{\dot{F}^{s}_{p,q}}.
    $$
    \item If $ q \in [2,\infty)$ then 
    $$
 (1-s)^{-1/q} [f]_{\dot{F}^{s}_{p,q}}  \lesssim_{n,p,q,\theta} [f]_{W^{s}_{p,q}} \lesssim_{n,p,q,\theta} (1-s)^{-1/2}[f]_{\dot{F}^{s}_{p,q}} .
    $$
\end{enumerate}
\end{thm}
We also have results for $\dot{F}^{s}_{p,2}$. 
\begin{thm}\label{1.10}
Let $ p,q \in (1,\infty)$ and $ \theta \in (0,1)$. Assume further that $ p > \frac{nq}{n+\theta q}$ and $ p' > \frac{nq'}{n+\theta q'}$. Then we have for $ s \in (\theta,1)$ the following 
\begin{enumerate}
    \item if $ q \in (1,2]$ then we have 
    $$
  (1-s)^{-1/q}  [f]_{\dot{F}^{s}_{p,2}} \lesssim_{n,p,q,\theta} [f]_{W^{s}_{p,q}}.
    $$
    \item If $ q \in [2,\infty)$ then we have
    $$
    [f]_{W^{s}_{p,q}} \lesssim_{n,p,q,\theta} (1-s)^{-1/q}[f]_{\dot{F}^{s}_{p,2}}.
    $$
\end{enumerate}
\end{thm}
The Sobolev estimates also hold in this setting. 
\begin{thm}\label{T1.11}
Assume $ p,q \in (1,\infty)$ with $ p > \frac{nq}{n+\theta q}$ for some $ \theta \in (0,1)$. Then for $ s \in (\theta,1)$ and $ v,\sigma \in [0,1]$ satisfying 
$$
v < s < \sigma.
$$
We get 
$$
[f]_{W^{s}_{p,q}} \lesssim_{n,p,q,\theta} (\sigma - s)^{-1/q} [f]_{\dot{F}^{\sigma}_{p,2}} + ( s-v)^{-1/q} [f]_{\dot{F}^{v}_{p,2}}.
$$
\end{thm}
As a consequence of these results we get BBM theorems for $ W^{s}_{p,q}$. BBM1 is already proven in \cite{Mo}. BBM2 seems new
\begin{cor}\label{BWs}
Let $p,q \in (1,\infty)$ with $ p' > \frac{nq'}{n+q'}$. Then assume $ \{ f_{k}\} \in \SSS(\R^{n})$ converges weakly to $f$ in $L^{p}$ and satisfies, for some sequence $\{ s_{k} \}$ that converges to $ 1$ from below, the following
$$
\Lambda = \sup_{k} \left( \Vert f_{k} \Vert_{L^{p}} + (1-s_{k})^{1/q} [f]_{W^{s}_{p,q}} \right) <\infty.
$$
Then we have that $ f \in W^{1,p}$ and 
$$
\Vert f \Vert_{W^{1,p}} \lesssim_{n,p,q} \Lambda.
$$
\end{cor}
\subsection{Relation between $W^{s}_{p,q}$ and $E^{s}_{p,q}$.} In light of the results above, one could be lead to believe that $ E^{s}_{p,q} \approx_{n,p,q} W^{s}_{p,q}$ when $ s \in (1/2,1)$ and $ p $ is "close" to $q$. We show that this is indeed the case when $ q = 2 $. We also show one direction is true when $ p > \frac{qn}{n+sq} $. 
\begin{thm}\label{A1}
Let $ \theta \in (0,1)$ and $ p > \frac{2n}{n+2\theta}$ and $ p' > \frac{2n}{n+2 \theta}$. Then we have
$$
[f]_{W^{s}_{p,2}} \approx_{n,p,q,\theta} [f]_{E^{s}_{p,2}},
$$
for $ s \in (\theta ,1 )$.
\end{thm}
We also have
\begin{thm}\label{A2}
Let $ p,q \in (1,\infty)$ and $ \theta \in (0,1)$ and $ p > \frac{nq}{n+\theta q}$. Then we have
$$
[f]_{W^{s}_{p,q}} \lesssim_{n,p,q,\theta} [f]_{E^{s}_{p,q}},
$$
for any $ s \in (\theta ,1)$.
\end{thm}

\subsection*{Some remarks and future directions} We mention a few remarks regarding the above results, and some open questions. 
\begin{rmk}[Sharpness of the constants]
The inequalities in Theorem \ref{q12} are, in some sense, sharp in $s$. To make this precise, let $ q \in (1,2)$ then an inequality of the form 
\begin{equation}\label{ed1}
(1-s)^{-1/q} [f]_{\dot{F}^{s}_{p,q}} \lesssim_{n,p,q} [f]_{E^{s}_{p,q}}
\end{equation}
for $s\in(0,1)$ and any $ f \in \SSS(\mathbb{R}^{n})$ does not hold. Indeed, this is a consequence of Theorem \ref{BBM1PR}. Assume inequality \eqref{ed1} is true, then by taking the limit $s \to 1^{-}$ and using Theorem \ref{BBM1PR} we have the following
$$
[f]_{\dot{F}^{1}_{p,q}} \lesssim_{n,p,q} [f]_{\dot{F}^{1}_{p,2}},
$$
which is not true since $q \in (1,2)$. Similarly, for $ q \in (2,\infty)$ we cannot have an inequality of the form 
\begin{equation}\label{ed2}
[f]_{E^{s}_{p,q}} \lesssim_{n,p,q} (1-s)^{-1/q} [f]_{\dot{F}^{s}_{p,q}}
\end{equation}
for all $ s \in (0,1)$ and $ f \in \SSS(\R^{n})$. Indeed, if \eqref{ed2} is true, then upon taking the limit $ s \to 1^{-}$ and using Theorem \ref{BBM1PR} we obtain 
$$
[f]_{\dot{F}^{1}_{p,2}} \lesssim_{n,p,q} [f]_{\dot{F}^{1}_{p,q}}
$$
which is again, not true because $ q \in (2,\infty)$. Also it is clear that an inequality of the form
\begin{equation}\label{ed3}
[f]_{E^{s}_{p,q}} \lesssim_{n,p,q} (1-s)^{-\beta} [f]_{\dot{F}^{s}_{p,q}}
\end{equation}
for all $ s \in (0,1)$, $ q \in (1,2)$, and some $ \beta < 1/q$ does not hold true. We again multiply both sides by $ (1-s)^{1/q}$ and take the limit as $s\to 1^{-}$ to obtain, using Theorem \ref{BBM1PR}, 
$$
[f]_{\dot{F}^{1}_{p,2}} \leq 0
$$
which is false. Similar ideas can be used to show that the estimates in Theorem \ref{QA1} are sharp in $s$. 
\end{rmk}
\begin{rmk}
The condition $ p > \frac{nq}{n+sq}$ in Theorem \ref{1.9} is necessary because as it was recently shown by Agarwal, Koskela, and Mohanta, see \cite[Corollary 1.4]{FKM}, the space $W^{s}_{p,q}$ consists only of constant functions when $ p \leq \frac{nq}{n+sq}$, and $p,q \in (1,\infty)$. On the other hand, the condition that $ p' > \frac{nq'}{n+sq'}$ is needed to run a duality argument to obtain the lower bounds in Theorem \ref{1.9}. Therefore, it is not clear to us if the condition $ p ' > \frac{nq'}{n+sq'}$ is actually necessary. 
\end{rmk}

Another remark is about Theorem \ref{BBM1PR}.
In Theorem \ref{BBM1PR} we have $\lim_{s \to 1^{-}} [f]_{E^{s}_{p,q}} \approx_{n,p,q} \Vert \nabla f \Vert_{L^{p}}$. It is not clear if we can make this an equality, this leads to the following question. 
\begin{ques}
For $ q,p \in (1,\infty)$, is there a constant $ C = C(n,p,q)$ so that
$$
\lim_{s \to 1^{-}} (1-s)^{1/q} [f]_{E^{s}_{p,q}} = C(n,p,q) \Vert \nabla f \Vert_{L^{p}}?
$$
\end{ques}
Lastly,   
in light of the results obtained for $ W^{s}_{p,q}$ and $ E^{s}_{p,q}$ it seems reasonable to hope that the following identification is true
$$
[f]_{W^{s}_{p,q}} \approx_{n,p,q} [f]_{E^{s}_{p,q}},
$$
when $ s \in (1/2,1)$ and $ p > \frac{nq}{n+ qs}$ and $ p' > \frac{ nq'}{n+ sq'}$. Theorem \ref{A2} gives us one direction, but it is not clear to us how one can prove the other direction. 
\begin{ques}
Is it true that for $ s \in (1/2,1)$, and $ p,q \in(1,\infty)$ satisfying  $p > \frac{nq}{n+ qs} $, and $ p' >\frac{ nq'}{n+ sq'} $ one has 
$$
[f]_{E^{s}_{p,q}} \approx_{n,p,q} [f]_{W^{s}_{p,q}}?
$$
\end{ques}

\subsection*{Plan of paper} In Section 2 we recall the basic notions and lemmata needed to prove the main results. In section 3 we prove the results regarding $E^{s}_{p,q}$. Namely, we prove Theorem 1.3, Theorem 1.4, Theorem 1.5, Theorem 1.6, Theorem 1.7, and Corollary 1.8. The main two ingredients for the proofs of the results in section 3 are the characterization of Triebel-Lizorkin spaces in terms of the Poisson kernel given in \cite{CB}, and a Littlewood-Paley theorem for mixed $L^{p}$ spaces, Lemma \ref{Kformixed}.   
  In section 4, we prove the results for $W^{s}_{p,q}$. Namely, we prove Theorem 1.9, Theorem 1.10 , Theorem 1.11, and Corollary 1.12. This will be done by first proving Theorem \ref{A2}.  Then using the results we obtained in section 3 regarding $E_{p,q}^{s}$ we obtain all of the upper bounds appearing in Theorem 1.9, Theorem 1.10. The lower bounds appearing in Theorem 1.9, Theorem 1.10, will be proven by duality, similar to what has been done in \cite{P}. 
\subsection*{Acknowledgment} I owe a special gratitude to A.Schikorra for introducing me to the paper \cite{P}. And I would like to thank him for sharing his ideas and expertise on how one might generalize the results there for $W^{s}_{p,q}.$ 

The research that has led to these results was funded by the National Science Foundation (NSF), Career DMS-2044898 (PI: Schikorra), and the NSF grant DMS-2055171 (PI:Haj{\l}asz).

\section{Preliminaries}
\subsection*{Notation}
The notation is fairly standard. We denote by $C^{\infty}_{0}(\Omega)$ the space of all smooth functions on $\Omega$ with compact support and $ \SSS(\R^{n})$ the class of Schwartz functions. Also we denote by $L^{p}(\Omega)$ the space of measurable functions with
$$
\int_{\Omega} \vert f \vert^{p} < \infty.
$$
And $W^{1,p}(\Omega)$ is the space of all $L^{p}(\Omega)$ functions with distributional derivatives of order $ 1$ being also in $L^{p}(\Omega)$. When $ \Omega = \mathbb{R}^{n}$ we will always write $L^{p}$ and $W^{1,p}$ and $C^{\infty}_{0}$. The average value of a function $ f \in L^{1}_{loc}$ over some measurable set $ A $ is written by 
$$
\dashint_{A} f(x) dx : = \frac{1}{|A|} \int_{A} f(x) dx. 
$$
For an exponent $ p \in (1,\infty)$ we write $ p'$ to mean the  H\"{o}lder conjugate, that is the number satisfying $ \frac{1}{p} + \frac{1}{p'} = 1$. The Fourier transform of a function $ f $ is denoted by either $ \hat{f}$ or $ \mathcal{F}(f)$, where we follow the convention 
$$
\hat{f}(\xi) : = \int_{\R^{n}} e^{ 2 \pi i \xi \cdot x} f(x) dx.
$$
Furthermore, given $ f \in L^{1}_{loc}$ we define $ \mathcal{M}(f)$ to be the Hardy-Littlewood maximal function. That is, 
$$
\mathcal{M}(f) (x) : = \sup_{r > 0} \dashint_{B(x,r)} | f(y) | dy.
$$
Lastly, whenever we write $ A \lesssim B $ we mean there is a constant $ C $ so that $ A \leq C B$. To indicate the dependence of $ C $ we will usually include the parameters in which $C$ depends on, for example if $ C = C (n,p)$ then we write $ A \lesssim_{n,p} B$. If we do not include the parameters then that would mean $C$ only depends on the dimension $n$, or other universal parameters. If $ A \lesssim B $ and $ B \lesssim A $ then we will write $ A \approx B$.

\subsection{Function spaces} 
In this subsection we recall the basic definitions needed for the spaces under consideration. First, take $ \varphi$ to be a Schwartz function with the property that $ \mathcal{F}(\varphi)$ is supported in the annulus $ B(0,2) \setminus B(0,1)$ and $ \mathcal{F}(\varphi) =1 $ in $B(0,7/4) \setminus B(0,5/4)$. Then denote by $ \varphi_{j}$ to mean 
$$
\varphi_{j}(x) = 2^{-jn} \varphi(2^{j} x).
$$
Then for a tempered distribution $f$, we define the Littlewood-Paley projections as 
$$
\Delta_{j}f(x) = \varphi_{j}* f(x).
$$
With this notation we can define the homogeneous Triebel-Lizorkin space.
\begin{defn}\label{DF}
Let $ s \in \mathbb{R}$ and $ p,q \in [1, \infty)$. Then we define $ \dot{F}^{s}_{p,q}$ as the space of all tempered distributions $f$ so that 
$$
[f]_{\dot{F}^{s}_{p,q}} = \left(\int_{\mathbb{R}^{n}} \left(\sum_{j \in \mathbb{Z}} 2^{ s q j}\vert \Delta_{j} f(x) \vert^{q} \right)^{p/q} dx \right)^{1/p} < \infty.
$$
\end{defn}
Next, we define the inhomogeneous Triebel-Lizorkin space. First let $ \eta $ be a function with the property that its Fourier transform is supported in the ball $B(0,1)$ with $ \hat{\eta} = 1 $ in $B(0,1/2)$. Then we define 
$$
\Delta_{\leq 0} f(x) = \eta * f (x). 
$$
\begin{defn}
Let $ s \in \mathbb{R}$. And $ p,q \in [1,\infty)$ then we define the inhomogeneous Triebel-Lizorkin space as the space of all tempered distributions $ f $ so that 
$$
\Vert f \Vert^{p}_{F^{s}_{p,q}} = \int_{\mathbb{R}^{n}}\left( \vert \Delta_{\leq 0 } f \vert^{q} + \sum_{j \geq 1} 2^{jsq} \vert \Delta_{j} f \vert^{q} \right)^{p/q} dx < \infty.
$$
\end{defn}
For further references regarding the Triebel-Lizorkin space, see \cite{GMF} and \cite{HT}. 
For our purposes, it would be useful to have an explicit kernel $ \varphi_{j}$, specifically a description of $ \varphi_{j}$ in terms of the Poisson kernel. Fortunately, this has already been done in \cite{CB}. Before we state their result, let us recall the definition and properties of the Poisson kernel.
\begin{defn}\label{defpoission}
Let $ t > 0$. Then we define the Poisson kernel $P_{t}$ as 
$$
P_{t}(x)= c_{n}\frac{t}{(t^{2} +\vert x \vert^{2})^{\frac{n+1}{2}}},
$$
where $ c_{n}$ is chosen so that $P_{t}$ has integral $ 1 $. 
\end{defn}

\begin{lemma}\label{Psn}
Let $ f \in L^{p}$ for some $p \in[1,\infty)$. Then the following is true
\begin{enumerate}
    \item $P_{t} * f$ is a harmonic function in $ \mathbb{R}^{n+1}_{+}$, as a function of $ (x,t)$;
    \item $P_{t} * f \to f $ as $ t \to 0$ in $L^{p}$;
    \item $ \mathcal{F}_{x} (P_{t})(\xi) = c_{n} e^{- 2 t \pi \vert \xi \vert} $;
    \item $ \mathcal{F}_{x}(\partial^{m}_{t} P_t)(\xi) = c_{n,m} \vert \xi \vert^{m}e^{- 2t\pi \vert \xi \vert} $;
    \item for $ t,r > 0 $ we have $ P_{r} * P_{t} *f (x) = P_{t+r}* f(x)$;
    \item Take $ f \in \SSS(\R^{n})$ and $ r > 0 $, then 
    \begin{equation}\label{maximalpoissonest}
    \vert \partial^{2}_{t} P_{t} * f (x) \vert \lesssim_{n} \mathcal{M} ( \vert F(\cdot,r) \vert )
    \end{equation}
    for all $ t \in (r,2r)$,
    where $ F(x,r) = \partial^{2}_{t} P_{r} * f(x)$.
\end{enumerate}
\end{lemma}
\begin{proof}
(1),(2),(3),(4) and (5) can be found in \cite[Page 62]{StS}. We prove (6), by definition of the Poisson integral we have for $ t \in (r,2r)$
\begin{equation*}
\begin{split}
& \partial^{2}_{t}P_{t}*f(x) = 4 \pi^{2} \int_{\mathbb{R}^{n}} e^{2 \pi i \xi . x} \vert \xi \vert^{2} e^{- 2 \pi t \vert \xi \vert} \hat{f}(\xi ) d\xi
\\
& = 4 \pi^{2} \int_{\mathbb{R}^{n}} e^{2 \pi i \xi . x} e^{- 2 \pi (t-r) \vert \xi \vert} ( \vert \xi \vert^{2} e^{-2 \pi r \vert \xi \vert}
\hat{f}(\xi )) d\xi
\\
& = P_{(t-r)} * ( \partial^{2}_{t} P_{r} * f)(x).
\end{split}
\end{equation*}
Therefore, 
$$
\vert \partial^{2}_{t}P_{t}*f(x)\vert \lesssim \mathcal{M}( \partial^{2}_{t} P_{r} * f )(x). 
$$

\end{proof}

Denote by $ \phi_{j}$ the function with Fourier transform $ \vert 2^{-j} \xi \vert^{2} e^{-2 \pi 2^{-j} \vert \xi \vert}$. In other words $ \phi_{j} = c 2^{-2j} \partial^{2}_{t} P_{2^{-j}}$. Then if one checks the dependency of the differentiation parameter $s$ in the proof of \cite[Theorem 1.4]{CB} we get the following
\begin{lemma}\label{PT}
Let $ f \in \SSS(\R^{n})$, $p,q \in(1,\infty)$ and $ s \in[0,1]$. Then
$$
[f]_{\dot{F}^{s}_{p,q}} \approx_{n,p,q} \left(\int_{\mathbb{R}^{n}} \left(\sum_{j \in \mathbb{Z}} \vert 2^{js} \phi_{j}*f(x) \vert^{q} \right)^{p/q} dx \right)^{1/p}.
$$
\end{lemma}
The point here is that there is no dependency on $ s$ when $ s \in [0,1]$. In fact, Candy and Bui prove more in \cite{CB}. They prove the lemma above is valid for all $ s \in [0,2)$ but the constant blows up as $ s \to 2$. Loosely speaking, the more derivatives we take on the Poisson kernel the more Triebel-Lizorkin spaces we cover. For example, if we took our kernel to be $\phi_{j}^{*}$ so that $ \mathcal{F}(\phi_{j}^{*}) = \vert 2^{-j} \xi \vert^{3} e^{-2^{-j}\vert \xi \vert}$, then the results in \cite{CB} tell us that we can use this as the kernel in the definition of Triebel space, and this will be equivalent to the usual Triebel Lizorkin space with constant $C_{s}$ that blows up only if $ s \to 3$. In other words, there is no dependency on $ s$ when $s \in [0,2]$ and $ p,q \in(1,\infty)$. In our case, we only care about $ s \in (0,1)$ so $ \phi_{j}$ is enough. The proof of Lemma \ref{PT} can be found in the appendix.

Another fractional space that we need is the following fractional homogeneous Sobolev space.
\begin{defn}
let $ s \in \mathbb{R}$, and $ p\in (1,\infty)$ then we define the space $ \dot{H}^{s}_{p}$ as the space of all measurable functions $ f  $ so that
$$
\Vert  \mathcal{F}^{-1}(\vert \cdot \vert^{s} \hat{f}) \Vert_{L^{p}} <\infty. 
$$

\end{defn}
The next lemma connects $\dot{F}^{s}_{p,q}$ to various other spaces. For a reference, see \cite{C}, \cite{GC}, \cite{GMF} and \cite{HT}. 
\begin{lemma}
Let $ s \in (0,1)$ and $ p,q \in (1,\infty)$ then we have for any $ f \in \SSS(\R^{n})$
\begin{enumerate}
    \item $$
[f]_{E^{s}_{p,q}} \approx_{n,p,q,s} [f]_{\dot{F}^{s}_{p,q}} 
$$
\item 
$$
\Vert f \Vert_{L^{p}} \approx_{p,n} [f]_{\dot{F}^{0}_{p,2}}
$$
\item 
$$
\Vert f \Vert_{\dot{H}^{s}_{p}} \approx_{n,p,s} [f]_{\dot{F}^{s}_{p,2}} 
$$
\item  if $ p > \frac{ nq}{n+sq}$, in particular if $ q=p$ then
$$
[f]_{W^{s}_{p,q}} \approx_{s,p,q,n} [f]_{\dot{F}^{s}_{p,q}}.
$$
\end{enumerate}

\end{lemma}
In the proof of BBM2 we will need the following lemma, which can be found in \cite[Lemma 2.7]{P}.
\begin{lemma}\label{WTL}
Let $ p,q \in (1,\infty)$. Assume $ \{ f_{k} \} \in L^{p}$ converges weakly to some $ f \in L^{p}$, and $ s_{k} \to t \in (0,\infty)$ from below. Assume further that 
$$
\sup_{k} [f_{k}]_{\dot{F}^{s_{k}}_{p,q}} \leq A < \infty.
$$
Then one has 
$$
[f]_{\dot{F}^{t}_{p,q}} \leq A.
$$
\end{lemma}

Another spaces that will play an important part in proving the lower bounds for $ E^{s}_{p,q}$ are the mixed $L^{p}$ spaces
\begin{defn}
Let $ p,q \in [1,\infty]$. Then define $ L^{p}(l^{q})$ as the space of all sequences of functions $ \{ f_{j} \}$ so that 
$$
\int_{\mathbb{R}^{n}} \Vert f_{j}(x) \Vert^{p}_{l^{q}} dx < \infty.
$$
\end{defn}
The following lemma is contained in \cite{MLP}
\begin{lemma}
Let $ p,q \in [1,\infty)$. Then $L^{p}(l^{q})$ is a Banach space with dual given by $ L^{p'}(l^{q'})$. In addition, the following equality holds 
$$
\Vert f_{j} \Vert_{L^{p}(l^{q})} = \sup_{ \substack{ g_{j} \in L^{p'}(l^{q'})  \\  \Vert g \Vert_{L^{p'}(l^{q'}) \leq 1}}} \int_{\mathbb{R}^{n}} \sum_{j \in \mathbb{Z}} f_{j} g_{j}.
$$
\end{lemma}
We also consider other mixed $L^{p}$ spaces
\begin{defn}\label{MixedLpS}
Let $ p,q \in (1,\infty)$ and $ \Omega \subset \mathbb{R}^{m}$ be open, then for $ \alpha : \Omega \to \mathbb{R}$ positive and continuous, we define $L_{\alpha}^{p,q}(\mathbb{R}^{n} \times \Omega)$ as follows
$$
L^{p,q}_{\alpha} (\mathbb{R}^{n} \times \Omega) = \{ f \in L^{1}_{loc}(\mathbb{R}^{n} \times \Omega) : \int_{\mathbb{R}^{n}} \left( \int_{\Omega} \vert f(x,r) \vert^{q} \alpha(r) dr \right)^{p/q} dx < \infty \}.$$
\end{defn}
The following lemma can be found \cite{MLP}.
\begin{lemma}\label{Lastdual}
Let $ p,q \in (1,\infty)$, $\Omega \subset \mathbb{R}^{m}$ open and $ \alpha : \Omega \to \mathbb{R}$. Then $L^{p,q}_{\alpha}(\mathbb{R}^{n} \times \Omega)$ is a Banach space under the norm 
$$
\Vert f \Vert^{p}_{L^{p,q}_{\alpha}} = \int_{\mathbb{R}^{n}} \left( \int_{\Omega} \vert f(x,r) \vert^{q} \alpha(r) dr \right)^{p/q} dx,
$$
with dual given by $L^{p',q'}_{\alpha}(\mathbb{R}^{n} \times \Omega).$
\end{lemma}
We will also need a duality characterization of the Triebel-Lizorkin space, but first we need to recall the definition of the fractional Laplace.
\begin{defn}
Let $ s \in \mathbb{R}$. Then we define $ (-\Delta)^{s/2}$ as the operator with symbol $ \vert \xi \vert^{s}$. That is 
$$
(-\Delta)^{s/2}(f) = \mathcal{F}^{-1}(\vert \cdot \vert^{s} \hat{f}).
$$
\end{defn}

For our purposes, a more useful representation is the following, see for example \cite[Lemma 2.2]{P}. 
\begin{lemma}\label{2.9}
Let $ s \in (0,2)$ and $ f \in \SSS(\R^{n})$. Then
$$
(-\Delta)^{s/2} f(x) \approx C(n,s) \int_{\mathbb{R}^{n}} \frac{ 2 f(x)- f(x+z) -f(x-z) }{\vert z \vert^{s+n}} dz,
$$    
where $ C(n,s) \approx_{n} \min\{s,1-s\}$.
\end{lemma}
The following duality characterization will play a crucial part in proving the results for $W^{s}_{p,q}$. For a proof see \cite[Theorem 2.8]{P}.
\begin{lemma}\label{2.10}
Let $ f \in \SSS(\R^{n})$, $p,q \in(1,\infty)$ and $ s\in(0,1)$. Then there exists $ g \in \SSS(\mathbb{R}^{n})$ so that $ [g]_{\dot{F}^{s}_{p',q'}} \leq 1$ and 
$$
[f]_{\dot{F}^{s}_{p,q}} \approx_{n,p,q} \vert \int_{\mathbb{R}^{n}} (-\Delta)^{s/2} f(x) \space\ (-\Delta)^{s/2}g(x) dx \vert.
$$
\end{lemma}
A consequence of Lemma \ref{2.10} and Lemma \ref{2.9} is the following result, the proof can be found in \cite[Proposition 6.1]{P}.
\begin{lemma}\label{DualityForW}
Let $ s \in (0,1)$ and $ p,q \in (1,\infty)$, denote by $ p'$ and $ q' $ the H\"{o}lder conjugates of $ p,q $ respectively, and take $ t_{1},t_{2}$ so that $ t_{1} + t_{2} = 2s$. Then we have for $ f\in  \SSS(\R^{n})$
\begin{equation}\label{dualfpqq}
[f]_{\dot{F}^{s}_{p,q}} \lesssim_{p,q,n} \min\{s,1-s\} [f]_{W^{t_{1}}_{p,q}} \sup_{ \substack{ g \in C^{\infty}_{0} \\ [g]_{\dot{F}^{s}_{p',q'}} \leq 1 }} [g]_{W^{t_{2}}_{p'.q'}}.
\end{equation}
Further, we have 
\begin{equation}\label{dualfpq}
[f]_{\dot{F}^{s}_{p,q}} \lesssim_{p,q,n} [f]_{\dot{F}^{0}_{p,q}}+  \min\{s,1-s\} [f]_{W^{t_{1}}_{p,q}}  \sup_{\substack{g \in C^{\infty}_{0}\\ \text{supp}(\mathcal{F}(g) ) \subset \{ \vert \xi \vert >1/4 \}\\ [g]_{\dot{F}^{s}_{p',q'}} \leq 1}}[g]_{W^{t_{2}}_{p'.q'}}.
\end{equation}
Moreover, both inequalities hold if we replace the left hand side by $\dot{F}^{s}_{p,2}$. Namely, 
\begin{equation}\label{dualfp2}
[f]_{\dot{F}^{s}_{p,2}} \lesssim_{p,q,n} \min\{s,1-s\} [f]_{W^{t_{1}}_{p,q}}  \sup_{ \substack{ g \in C^{\infty}_{0} \\ [g]_{\dot{F}^{s}_{p',2}} \leq 1 }} [g]_{W^{t_{2}}_{p'.q'}}.
\end{equation}
And
\begin{equation}\label{seconddualfp2}
[f]_{\dot{F}^{s}_{p,2}} \lesssim_{p,q,n} [f]_{\dot{F}^{0}_{p,q}}+  \min\{s,1-s\} [f]_{W^{t_{1}}_{p,q}}  \sup_{\substack{g \in C^{\infty}_{0}\\ \text{supp}(\mathcal{F}(g) ) \subset \{ \vert \xi \vert >1/4 \}\\ [g]_{\dot{F}^{s}_{p',2}} \leq 1}}[g]_{W^{t_{2}}_{p'.q'}}.
\end{equation}

\end{lemma}
The advantage of the second inequalities \eqref{dualfpq} and \eqref{seconddualfp2}, as oppose to \eqref{dualfpqq} and \eqref{dualfp2}, is the support of the Fourier transform of $ g$. 

We close this subsection with the following relation between the Poisson integral and the fractional Laplace. This result is a special case of \cite{Caf}.
\begin{lemma}\label{PC}
Let $P_{t}$ denote the Poisson kernel, and let $ f \in \SSS(\mathbb{R}^{n}).$ Then there exists a constant $C= C(n)$ so that for any $ x \in \mathbb{R}^{n}$ we have
$$
\lim_{t\to 0} \partial_{t} P_{t}*f(x) = C \Delta^{1/2} f(x).
$$
\end{lemma}

\subsection{Littlewood-Paley for mixed $L^{p}$ spaces} The key lemma for proving the upper bounds between $E^{s}_{p,q}$ and $\dot{F}^{s}_{p,q}$ is a Littlewood-Paley inequality for mixed $L^{p}$ spaces. First, let us recall the standard Littlewood-Paley theorem, see \cite[Page 82]{StS} for a proof.
\begin{lemma}\label{LP}
Let $p\in (1,\infty)$ and $ l \in \mathbb{N}$, then for any $ f \in L^{p}(\mathbb{R}^{n})$ we have
$$
\Vert f \Vert_{L^{p}} \approx_{n,p,l} \left( \int_{\mathbb{R}^{n}} \left( \int_{0}^{\infty} t^{2l-1} \vert \partial^{l}_{t} P_{t} * f (x) \vert^{2} dt \right)^{p/2} dx \right)^{1/p}.
$$
\end{lemma}
Next, we aim to have the same result but for mixed $L^{p}$ spaces. That is, assume $ \Omega \subset \mathbb{R}^{m}$ is open, and $ \alpha : \Omega \to \mathbb{R}$ is continuous positive function. Then for $ f \in L^{p,q}_{\alpha}(\mathbb{R}^{n} \times \Omega)$ we wish to have the following equivalence for $ p,q \in (1,\infty)$: 
\begin{equation*}
\begin{split}
&\int_{\mathbb{R}^{n}} \left( \int_{\Omega} \vert f(x,r) \vert^{q} \alpha(r) dr \right)^{p/q} dx
\\
&
\approx_{n,p,q,l} \int_{\mathbb{R}^{n}} \left( \int_{\Omega} \left( \int_{0}^{\infty} t^{2l-1} \vert \partial^{l}_{t} P_{t} * f(x,r) \vert^{2} dt\right)^{q/2} \alpha(r) dr \right)^{p/q} dx.
\end{split}
\end{equation*}
Notice that if $ p= q$ then this follows from Lemma \ref{LP} above and Fubini's theorem. However, to prove this for when $ p \neq q$, we need to use the theory of singular integrals for Banach valued functions. It is worth mentioning that a Littelwood-Paley theorem for mixed $L^{p}$ spaces is probably known; however since we are unable to find a reference for this result in the classical literature, we refer the reader to the appendix for a detailed proof.

\begin{lemma}\label{Kformixed}
Let $ p,q \in (1,\infty)$ and $ \Omega \subset \mathbb{R}^{m}$ open, and $ \alpha : \Omega \to \mathbb{R}$ a continuous positive function. Then for $ f \in L^{p,q}_{\alpha}(\mathbb{R}^{n} \times \Omega)$ we have the following 
\begin{equation*}
\begin{split}
&\int_{\mathbb{R}^{n}} \left( \int_{\Omega} \vert f(x,r) \vert^{q} \alpha(r) dr \right)^{p/q} dx
\\
&
\approx_{n,p,q,l} \int_{\mathbb{R}^{n}} \left( \int_{\Omega} \left( \int_{0}^{\infty} t^{2l-1} \vert \partial^{l}_{t} P_{t} * f(x,r) \vert^{2} dt\right)^{q/2} \alpha(r) dr \right)^{p/q} dx.
\end{split}
\end{equation*}
\end{lemma}
Essentially, Lemma \ref{Kformixed} tells us that we can apply the classical Littlewood-Paley without using Fubini's Theorem. To illustrate this, we include below two immediate consequences of Lemma \ref{Kformixed}.
\begin{cor}\label{NK1}
Let $ p,q \in (1,\infty)$ and $ F \in C^{\infty}(\mathbb{R}^{n}\times (0,\infty))$ and assume further that we have 
$$
\int_{\mathbb{R}^{n}} \left( \int_{0}^{\infty} \vert F(x,r) \vert^{q} \frac{dr}{r} \right)^{p/q} dx < \infty.
$$
Then we have 
$$
\int_{\mathbb{R}^{n}} \left( \int_{0}^{\infty} \vert F(x,r) \vert^{q} \frac{dr}{r} \right)^{p/q} dx \approx_{n,p,q} \int_{\mathbb{R}^{n}} \left( \int_{0}^{\infty} \left(\int_{0}^{\infty} t \vert \partial_{t} P_{t} *F(x,r)\vert^{2} dt \right)^{q/2} \frac{dr}{r} \right)^{p/q} dx.
$$
\end{cor}
\begin{proof}
We apply Lemma \ref{Kformixed} with $ l = 1 $ to the function $ F $ which is an element of $L_{\alpha}^{p,q}(\mathbb{R}^{n} \times \Omega)$ for $ \Omega = (0,\infty)$, and $ \alpha(r) = \frac{1}{r}$.
\end{proof}
In the same manner we have another consequence of Lemma \ref{Kformixed} 
\begin{cor}\label{NK2}
Let $ F : \mathbb{R}^{n} \times( \mathbb{R}^{n} \setminus \{ 0 \}) \to \mathbb{R}$ be smooth with the property that 
$$
\int_{\mathbb{R}^{n}} \left( \int_{\mathbb{R}^{n} \setminus \{0 \}} \vert F(x,z) \vert^{q} \frac{dz}{\vert z \vert^{n}} \right)^{p/q} dx < \infty.
$$
Then we have 
\begin{equation*}
\begin{split}
& \int_{\mathbb{R}^{n}} \left( \int_{\mathbb{R}^{n} } \vert F(x,z) \vert^{q} \frac{dz}{\vert z \vert^{n}} \right)^{p/q} dx
\\
& \approx_{n,p,q} \int_{\mathbb{R}^{n}} \left( \int_{\mathbb{R}^{n}} \left( \int_{0}^{\infty} t^{3} \vert \partial^{2}_{t} P_{t} * F(x,z) dt \right)^{q/2} \frac{dz}{\vert z \vert^{n}} \right)^{p/q} dx.
\end{split}
\end{equation*}
Where the convolution in the inner integral is in the variable in $ x$. 
\end{cor}
\begin{proof}
We apply Lemma \ref{Kformixed} with $ \Omega = \mathbb{R}^{n} \setminus \{ 0 \} $,  and $ \alpha(z) = \vert z \vert^{-n} $ and $ l =2$ to the function $ F $ which is an element of $L^{p,q}_{\alpha}(\mathbb{R}^{n} \times \mathbb{R}^{n} \setminus \{ 0 \})$ 
\end{proof}

\subsection{Maximal inequalities} In proving the results for $ W^{s}_{p,q}$ we will need two maximal operators

\begin{defn}\label{PMax}
Let $ f \in \SSS(\R^{n})$ then we define the Peetre maximal function $ f^{*}_{j}$ as 
$$
f^{*}_{j}(x) = \sup_{z} \frac{ \vert \phi_{j} * f (x+z)\vert}{ \vert 1+2^{j}z \vert^{\lambda}}.
$$
Where $ \phi_{j}$ is as before, the function with Fourier transform given by $ \vert \xi 2^{-j}\vert^{2} e^{- 2 \pi 2^{-j} \vert \xi \vert}$. And $ \lambda$ is some free parameter, usually we require it to  satisfy $ \lambda > \max\{ n/p, n/q\}$ where $ p,q$ are the parameters in the Triebel-Lizorkin space.
\end{defn}
There is an analogous result of Lemma \ref{PT} for the above Peetre maximal function. The following lemma is taken from \cite[Theorem 1.4]{CB}.
\begin{lemma}\label{PetreeM}
Let $ p,q \in (1,\infty)$ and $ \lambda > \max\{ n/p, n/q\}$ and $ s\in (0,1)$. Then we have for any $ f \in L^{p}$
$$
\int_{\mathbb{R}^{n}} \left( \sum_{j \in \mathbb{Z}} \vert 2^{js} f^{*}_{j}(x) \vert^{q} \right)^{p/q} dx \lesssim_{n,p,q,\lambda} [f]^{p}_{\dot{F}^{s}_{p,q}}.
$$
\end{lemma}
Next, we will need the following result, which can be found in \cite{SC}. 
\begin{lemma}\label{SteinF}
Let $ p ,q \in (1,\infty)$. Then we we have for $ f \in L^{p}(l^{q})$
$$
\int_{\mathbb{R}^{n}} \left( \sum_{j \in \mathbb{Z}} \mathcal{M}(f_{j})(x)^{q} \right)^{p/q} dx \lesssim_{p,q,n} \int_{\mathbb{R}^{n}} \left( \sum_{j\in \mathbb{Z}} \vert f_{j}(x)\vert^{q} \right)^{p/q} dx.
$$
Where $ \mathcal{M}(\cdot)$ is the Hardy-Littlewood maximal operator.
\end{lemma}
We will also need a generalized version Lemma \ref{SteinF}. The following result can be found in \cite[Theorem 2.12]{CMax}.
\begin{lemma}\label{CSF}
Let $ m \in \mathbb{N}$ and take $ q_{1},...,q_{m} \in (1,\infty)$ and $ f \in L^{1}_{loc}( \mathbb{R}^{n} \times \mathbb{R}^{m})$. Then define 
$$
f^{*}(x,z) = \sup_{\sigma > 0} \dashint_{B(x,\sigma)} \vert f(y,z) \vert dy.
$$
Further, define 
$$
T_{\bar{q}} (f) (x) =  \left( \int_{-\infty}^{\infty}  \dots \big[ \int_{-\infty}^{\infty} \vert f(x,t_{1},...,t_{m}) \vert^{q_{1}} dt_{1} \big]^{\frac{q_{2}}{q_{1}}} \dots dt_{m} \right)^{1/q_{m}}.
$$
Then we have for any $ p \in (1,\infty)$ the following
$$
\Vert T_{\bar{q}} (f^{*}) \Vert_{L^{p}} \leq C(n,p,q_{1} \dots q_{m}) \Vert T_{\bar{q}} (f) \Vert_{L^{p}}.
$$
\end{lemma}
\begin{rmk}
Before proceeding further, let us explain how Lemma \ref{CSF} is useful in our setting. A special case of it is the following inequality, which holds for $ f \in L^{1}_{loc} ( \mathbb{R}^{n} \times (0,\infty))$
$$
\int_{\mathbb{R}^{n}} \left( \int_{0}^{\infty} \vert \mathcal{M}( f(.,t))(x) \vert^{q} dt \right)^{p/q} dx \lesssim_{n,p,q}  \int_{\mathbb{R}^{n}} \left( \int_{0}^{\infty} \vert  f(x,t)\vert^{q} dt \right)^{p/q} dx.
$$
And more crucially, we have the same for functions in $ f \in L^{1}_{loc} ( \mathbb{R}^{n} \times ( 0,\infty)^{2})$. Namely,
\begin{equation*}
\begin{split}
& \int_{\mathbb{R}^{n}} \left( \int_{0}^{\infty}\big[ \int_{0}^{\infty}\vert \mathcal{M}( f(.,t,r))(x) \vert^{2} dr\big]^{q/2} dt  \right)^{p/q} dx 
\\
&
\lesssim_{n,p,q}  \int_{\mathbb{R}^{n}} \left(\int_{0}^{\infty} \big[\int_{0}^{\infty} \vert  f(x,t,r)\vert^{2} dr\big]^{q/2} dt \right)^{p/q} dx.
\end{split}
\end{equation*}
This together with Lemma \ref{Kformixed} will play a crucial part in relating the space $ W^{s}_{p,q} $ to $ E^{s}_{p,q}$.
\end{rmk}

Next, using Lemma \ref{CSF} and Lemma \ref{PT}, we obtain the following equivalence. The importance of the next lemma is that the inequalities do not depend on $s$. 
\begin{lemma}\label{DTC}
Let $ p,q \in (1,\infty)$. Then for $ f \in \SSS(\R^{n})$ we have 
$$
\int_{\mathbb{R}^{n}} \left( \int_{0}^{\infty} t^{q (2-s)} \vert \partial^{2}_{t} P_{t} *f (x) \vert^{q} dt/t \right)^{p/q}  dx \approx_{n,p,q} [f]^{p}_{\dot{F}^{s}_{p,q}}.
$$
\end{lemma}
\begin{proof}
Put $ f_{j} = \phi_{j} * f $ where $ \phi_{j} = 2^{-2j} \partial^{2}_{t} P_{2^{-j}}$.
Then partition the integral in $ t$ for when $ 2^{-j-1} \leq t \leq 2^{-j}$. Then we have by Lemma \ref{Psn}, \eqref{maximalpoissonest}
\begin{equation*}
    \begin{split}
        & \int_{0}^{\infty} t^{q(2-s)} | \partial^{2}_{t} P_{t} * f(x) |^{q} \frac{dt}{t} 
        \\
        & \approx \sum_{j \in \mathbb{Z}} 2^{-jq (2-s)} 2^{j} \int_{2^{-j-1}}^{2^{-j}} | \partial^{2}_{t} P_{t} * f(x) |^{q} dt
        \\
        & \overset{\eqref{maximalpoissonest}}{\lesssim}\sum_{j \in \mathbb{Z}} 2^{sq j} 2^{-2q j}|\mathcal{M}(\partial^{2}_{t} P_{2^{-j-1}} * f )(x)|^{q}
        \\
        & \lesssim \sum_{j \in \mathbb{Z}} 2^{jsq} |\mathcal{M}( f_{j} )(x)|^{q}
    \end{split}
\end{equation*}
where we used the definition of $ f_{j}$ in the last inequality. 
Therefore, raising to power $ p/q$ and integrating in $ x $ gives us 
\begin{equation*}
\begin{split}
& \int_{\mathbb{R}^{n}} \left( \int_{0}^{\infty} t^{q (2-s)} \vert \partial^{2}_{t} P_{t} *f (x) \vert^{q} dt/t \right)^{p/q}
\\
& \lesssim \int_{\mathbb{R}^{n}} \left( \sum_{j \in \mathbb{Z}} 2^{qjs} \vert \mathcal{M}(f_{j})(x) \vert^{q} \right)^{p/q} dx
\lesssim_{n,p,q} [f]^{p}_{\dot{F}^{s}_{p,q}}.
\end{split}
\end{equation*}
Where we used Lemma \ref{SteinF} and Lemma \ref{PT}.
Next, we prove the other direction. By Lemma \ref{PT}, and \eqref{maximalpoissonest}, we have 
\begin{equation*}
\begin{split}
& [f]^{p}_{\dot{F}^{s}_{p,q}}  \lesssim_{n,p,q} \int_{\mathbb{R}^{n}} \left( \sum_{j} 2^{jsq} \vert f_{j}(x) \vert^{q} \right)^{p/q} dx
\\
& \approx \int_{\mathbb{R}^{n}} \left( \sum_{j} \int_{2^{-j}}^{2^{-j+1}} 2^{jsq} 2^{j} \vert f_{j}(x) \vert^{q} dt \right)^{p/q} dx
\\
& \approx \int_{\mathbb{R}^{n}} \left( \sum_{j} \int_{2^{-j}}^{2^{-j+1}} t^{-sq}  \vert f_{j}(x) \vert^{q} dt/t \right)^{p/q} dx
\\
& = \int_{\mathbb{R}^{n}} \left( \sum_{j} \int_{2^{-j}}^{2^{-j+1}} t^{-s q} | 2^{-2j} \partial^{2}_{t} P_{2^{-j}} * f(x) |^{q} \frac{dt}{t} \right)
\\
& \overset{\eqref{maximalpoissonest}}{\lesssim}
 \int_{\mathbb{R}^{n}} \left( \sum_{j} \int_{2^{-j}}^{2^{-j+1}} t^{-sq}  \vert \mathcal{M}( t^{2} \partial^{2}_{t} P_{t/2} *f )(x) \vert^{q} dt/t \right)^{p/q} dx
\\
& \lesssim_{n,p,q} \int_{\mathbb{R}^{n}} \left(  \int_{0}^{\infty} t^{q(2-s)}  \vert \mathcal{M}( \partial^{2}_{t} P_{t/2} *f )(x) \vert^{q} dt/t \right)^{p/q} dx.
\end{split}
\end{equation*}
Now we use Lemma \ref{CSF} to conclude. 

\end{proof}

\section{Results for $ E_{p,q}^{s}$}
Throughout this section for $ f \in \SSS ( \R^{n}) $ we define $ f_{j} = \phi_{j} * f $ where $ \phi_{j} $ is the function with Fourier transform given by $ \vert 2^{-j} \xi \vert^{2} e^{-2 \pi 2^{-j} \vert \xi \vert}.$ That is, $ f_{j} = 2^{-2j} \partial^{2}_{t} P_{2^{-j}} * f (x)$. We will also use the following obvious result; for $ \sigma \in (0,2) $ and $ p > 0 $ one has
\begin{equation}\label{elementaryest1}
\sum_{j\geq 0} 2^{-j\sigma p} \approx_{p} \frac{1}{\sigma}.
\end{equation}
Similarly,
\begin{equation}\label{elementaryest2}
\sum_{j \leq 0} 2^{j \sigma p } \approx_{p} \frac{1}{\sigma}.
\end{equation}
Indeed, to see the first inequality, it suffices to prove
\begin{equation*}
\begin{split}
& | \frac{ \sigma}{1 -2^{ - \sigma p}} | \approx_{p} 1
\end{split}
\end{equation*}
which follows since 
$$
\lim_{\sigma \to 0} \frac{ \sigma}{1- 2^{-\sigma p}} = c_{p} \approx 1.
$$
The other inequality is proven in the same manner.\\
We will need one more elementary fact about the Gamma function, $\Gamma : (0,\infty) \to \mathbb{R}$, given by 
$$
\Gamma(r) : = \int_{0}^{\infty} t^{r-1} e^{-t} dt.
$$
It is easy to prove that for any $ r \in (0,1)$ we have 
\begin{equation}\label{gammafunctionbound}
\Gamma(r) \approx r^{-1}
\end{equation}
Indeed, this follows since $ \Gamma(1+r) = r \Gamma(r)$ and $ \Gamma \approx 1$ on the interval $[1,2]$.\\
\subsection{Upper bounds I}  
 The main results of this subsection are the upper bounds appearing in Theorem \ref{q12} and Theorem \ref{QA1}
\begin{prop}\label{U}
Let $ f \in \SSS(\R^{n}) $ and $ p,q \in (1,\infty)$ then one has 
\begin{enumerate}
    \item if $ q\in (1,2]$ then 
    \begin{equation}\label{upperboundforE1}
        [f]_{E_{p,q}^{s}} \lesssim_{n,q,p} (1-s)^{-1/q}[f]_{F^{s}_{p,q}}.
        \end{equation}
    \item If $ q \in [2, \infty)$ then we have 
    \begin{equation}\label{upperboundforE2}
        [f]_{E_{p,q}^{s}} \lesssim_{n,q,p} (1-s)^{-1/2}[f]_{F^{s}_{p,q}}.
    \end{equation}
     \item  If $ q \in [2, \infty) $ then we have 
     \begin{equation}\label{upperboundforE3}
         [f]_{E_{p,q}^{s}}  \lesssim_{n,p,q} (1-s)^{-1/q}[f]_{F^{s}_{p,2}}.
     \end{equation}
\end{enumerate}

\end{prop}
The proof of the upper bounds will be a sequence of propositions
\begin{prop}\label{M}
Let $ p,q \in (1,\infty)$.
Then for any $ f \in \SSS(\R^{n}) $ we have 
$$
[f]^{p}_{E^{s}_{p,q}} \approx_{n,p,q} \int_{\mathbb{R}^{n}} \left( \int_{0}^{\infty} r^{q(1-s)-1} \big[ \int_{r}^{\infty} t \vert \partial^{2}_{t} P_{t} * f (x) \vert^{2} dt \big]^{q/2} dr \right)^{p/q} dx.
$$

\end{prop}
\begin{proof} 
We will use Corollary \ref{NK1} on the function $ G(x,r) = r^{(1-s)} \partial_{r}P_{r} * f(x)$. Notice that this function is smooth and 
$$
\int_{\mathbb{R}^{n}} \left( \int_{0}^{\infty} \vert G(x,r) \vert^{q} \frac{dr}{r} \right)^{p/q} dx = [f]^{p}_{E^{s}_{p,q}} <\infty.
$$
Therefore, the hypothesis of Corollary \ref{NK1} is satisfied. Hence, we get
\begin{equation*}
\begin{split}
&[f]^{p}_{E^{s}_{p,q}} 
\\
& = \int_{\mathbb{R}^{n}} \left( \int_{0}^{\infty} r^{q(1-s)-1} | \partial_{r} P_{r} *f (x) |^{q} dr \right)^{p/q} dx
\\
& = \int_{\mathbb{R}^{n}} \left( \int_{0}^{\infty} \vert G(x,r) \vert^{q} \frac{dr}{r} \right)^{p/q} dx
\\
& \approx_{n,p,q} 
\int_{\mathbb{R}^{n}} \left(\int_{0}^{\infty} \left( \int_{0}^{\infty} t\vert \partial_{t} P_{t} * G(x,r) \vert^{2} dt \right)^{q/2} \frac{dr}{r} \right)^{p/q} dx
\\
&
= \int_{\mathbb{R}^{n}} \left( r^{q(1-s)-1} \big[ \int_{0}^{\infty} t | \partial_{r} \partial_{t} P_{r+t} * f(x)|^{2} dt\big]^{q/2} dr \right)^{p/q}.
\end{split} 
\end{equation*}
Next, we make a change of variable on the inner integral to get
\begin{equation}\label{refedit1}
\begin{split}
& [f]^{p}_{E^{s}_{p,q}}
\\
& \approx_{n,p,q} \int_{\mathbb{R}^{n}} \left( \int_{0}^{\infty}  r^{q(1-s)} 
 \big[ \int_{r}^{\infty} (t-r) \vert  \partial^{2}_{t} P_{t}*f(x) \vert^{2}  dt \big]^{q/2}  dr/r \right)^{p/q} dx,
\end{split}
\end{equation}
the right hand side of \eqref{refedit1} is bounded by 
 $$
 \lesssim \int_{\mathbb{R}^{n}} \left( \int_{0}^{\infty}  r^{q(1-s)} 
 \big[ \int_{r}^{\infty} t \vert  \partial^{2}_{t} P_{t}*f(x) \vert^{2}  dt \big]^{q/2}  dr/r \right)^{p/q} dx.
 $$
 This establishes the upper bound. Now for the lower bound
\begin{equation*}
\begin{split}
& [f]^{p}_{E^{s}_{p,q}} \\
& \overset{\eqref{refedit1}}{\approx_{n,p,q}}
  \int_{\mathbb{R}^{n}} \left( \int_{0}^{\infty}  r^{q(1-s)} 
 \big[ \int_{r}^{\infty} (t-r) \vert  \partial^{2}_{t} P_{t}*f(x) \vert^{2}  dt \big]^{q/2}  dr/r \right)^{p/q} dx  \\
 & \geq   \int_{\mathbb{R}^{n}} \left( \int_{0}^{\infty}  r^{q(1-s)} 
 \big[ \int_{2r}^{\infty} (t-r) \vert  \partial^{2}_{t} P_{t}*f(x) \vert^{2}  dt \big]^{q/2}  dr/r \right)^{p/q} dx.
\end{split}
\end{equation*}
Note that since $ t \in (2r,\infty)$, then we have that $ t-r \geq t/2$. Plugging this back in gives 
\begin{equation*}
\begin{split}
& [f]^{p}_{E^{s}_{p,q}} \\
& \gtrsim_{n,p,q} \int_{\mathbb{R}^{n}} \left( \int_{0}^{\infty}  r^{q(1-s)} 
 \big[ \int_{2r}^{\infty} t \vert  \partial^{2}_{t} P_{t}*f(x) \vert^{2}  dt \big]^{q/2}  dr/r \right)^{p/q} dx 
 \\ & 
 \\
 & = 2^{-q(1-s)} \int_{\mathbb{R}^{n}} \left( \int_{0}^{\infty}  (2r)^{q(1-s)} 
 \big[ \int_{2r}^{\infty} t \vert  \partial^{2}_{t} P_{t}*f(x) \vert^{2}  dt \big]^{q/2}  dr/r \right)^{p/q} dx
 \\
 & \gtrsim_{n,p,q} \int_{\mathbb{R}^{n}} \left( \int_{0}^{\infty}  r^{q(1-s)} 
 \big[ \int_{r}^{\infty} t \vert  \partial^{2}_{t} P_{t}*f(x) \vert^{2}  dt \big]^{q/2}  dr/r \right)^{p/q} dx.
 \end{split}
\end{equation*}
\end{proof}

\begin{prop}\label{geq1}
Assume $ p \in (1, \infty)$ and $ q \in (1,2]$. Then we have for $ f \in \SSS(\R^{n})$ 
$$
  \int_{\mathbb{R}^{n}} \left( \int_{0}^{\infty} r^{q(1-s)-1} \big[ \int_{r}^{\infty} t \vert \partial^{2}_{t} P_{t} * f (x) \vert^{2} dt \big]^{q/2} dr \right)^{p/q} dx \lesssim_{n,p,q}  (1-s)^{-p/q} [f]^{p}_{F^{s}_{p,q}}
 $$
\end{prop}
\begin{proof}
Fix $ x \in \mathbb{R}^{n} $. Then notice
\begin{equation*}
\begin{split}
& \int_{0}^{\infty} r^{q(1-s) - 1}\big[ \int_{r}^{\infty} t^{2} |\partial^{2}_{t} P_{t} * f (x)|^{2} \frac{dt}{t} \big]^{\frac{q}{2}} dr
\\
&\leq 
\sum_{k \in \mathbb{Z}} \int_{ 2^{-k-1}}^{2^{-k}} r^{q(1-s) -1} \big[ \sum_{j \leq k+C} \int_{2^{-j-1}}^{2^{-j}} | t\partial^{2}_{t} P_{t} * f (x)|^{2} \frac{dt}{t}\big]^{q/2} dr
\\
& \overset{\eqref{maximalpoissonest}}{\lesssim } \sum_{k \in \mathbb{Z}} 2^{-k q(1-s)} \big[  \sum_{j \leq k-C} 2^{-2j} | \mathcal{M} ( \partial_{t}^{2} P_{2^{-j-1}} * f )(x) |^{2} \big]^{q/2}
\\
& 
=\sum_{k \in \mathbb{Z}} 2^{-k q(1-s)} \big[  \sum_{j \leq k+C} 2^{-2j} 2^{4j} | \mathcal{M} ( 2^{-2j} \partial_{t}^{2} P_{2^{-j-1}} * f )(x) |^{2} \big]^{q/2}
\\
& \lesssim \sum_{k \in \mathbb{Z}} 2^{-k q(1-s)} \big[  \sum_{j \leq k+C} 2^{2j} |\mathcal{M} ( f_{j})(x) |^{2} \big]^{q/2}
\\
& \overset{\frac{q}{2} \leq 1}{ \lesssim } \sum_{k \in \mathbb{Z}} 2^{-k q(1-s)} \sum_{j \leq k+C} 2^{jq} |\mathcal{M}(f_{j} )(x)|^{q} 
\end{split}
\end{equation*}
Therefore, interchanging the sums in $j$ and $k$, we obtain 
\begin{equation*}
\begin{split}
& \int_{0}^{\infty} r^{q(1-s) - 1}\big[ \int_{r}^{\infty} t^{2} |\partial^{2}_{t} P_{t} * f (x)|^{2} \frac{dt}{t} \big]^{\frac{q}{2}} dr
\\
& \lesssim \sum_{j \in \mathbb{Z}} 2^{jq} | \mathcal{M} (f_{j} )(x) |^{q} \sum_{k \geq j -C} 2^{-kq (1-s)}
\\
& \overset{\eqref{elementaryest1}}{ \approx_{q}} (1-s)^{-1} \sum_{j \in \mathbb{Z}} 2^{jsq} | \mathcal{M} (f_{j} )(x) |^{q}
\end{split}
\end{equation*}
raising this to power $ \frac{p}{q}$ and integrating in $x$ gives us the result by Lemma \ref{SteinF}.

\end{proof}
 The above gives us the desired result for when $ q \in (1,2]$. Next we treat the case $ q \in [2, \infty)$. 

 \begin{prop}\label{geq2}
Let $ q \in [2, \infty)$ then we have for $ f \in \SSS(\mathbb{R}^{n})$ we have
$$
 \int_{\mathbb{R}^{n}} \left( \int_{0}^{\infty} r^{q(1-s)-1} \big[ \int_{r}^{\infty} t \vert \partial^{2}_{t} P_{t} * f (x) \vert^{2} dt \big]^{q/2} dr \right)^{p/q} dx 
 \lesssim_{n,p,q}  {(1-s)^{-p/2}} [f]_{\dot{F}^{s}_{p,q}}
$$
 \end{prop}
 \begin{proof}
Set $F(x,t) : = \partial^{2}_{t} P_{t} * f(x)$. Then notice that we have 
\begin{equation}\label{EQS1}
\begin{split}
&  \int_{0}^{\infty} r^{q(1-s)-1} \big[ \int_{r}^{\infty} t \vert \partial^{2}_{t} P_{t} * f (x) \vert^{2} dt \big]^{q/2} dr
\\
& =  \int_{0}^{\infty} r^{-qs-1} \big[ \int_{r}^{\infty} (rt)^{2} \vert F(x,t) \vert^{2} dt/t \big]^{q/2} dr
\\
& =  \int_{0}^{\infty} r^{-qs-1} \big[ \int_{1}^{\infty} \frac{1}{t^{2}}  \vert (rt)^{2} F(x,rt) \vert^{2} dt/t \big]^{q/2} dr.
\end{split}
\end{equation}
Where for the last equality we used the change of variable $ t \to r t $ on the inner integral. 
Next, we use \eqref{EQS1} and Minkowski's inequality with respect to the measures $ dr/r$ and $ dt/t$ to get 
\begin{equation*}
\begin{split}
& \left( \int_{0}^{\infty} r^{q(1-s)-1} \big[ \int_{r}^{\infty} t \vert F(x,t) \vert^{2} dt \big]^{q/2} dr \right)^{p/q} \\
& = \left(  \int_{0}^{\infty} r^{-qs-1} \big[ \int_{1}^{\infty} \frac{1}{t^{2}}  \vert (rt)^{2}  F(x,rt) \vert^{2} dt/t \big]^{q/2} dr \right)^{\frac{p}{2} \frac{2}{q}}
\\
& 
\leq \left(  \int_{1}^{\infty}  \big[ \int_{0}^{\infty} \frac{1}{t^{q}}r^{-qs} \vert (rt)^{2}  F(x,rt) \vert^{q} dr/r \big]^{2/q} dt/t \right)^{p/2}
\\
& 
= \left(  \int_{1}^{\infty}  \big[ \int_{0}^{\infty} \frac{1}{t^{q(1-s)}} \vert (rt)^{2-s} F(x,rt) \vert^{q} dr/r \big]^{2/q} dt/t \right)^{p/2}
\\
& = \left(  \int_{1}^{\infty} t^{-q(1-s) -1} \big[ \int_{0}^{\infty} \vert (rt)^{2-s} F(x,rt) \vert^{q} dr/r \big]^{2/q} dt \right)^{p/2}
\\
& = \left(  \int_{1}^{\infty} t^{-q(1-s) -1} \big[ \int_{0}^{\infty} \vert (r)^{2-s} F(x,r) \vert^{q} dr/r \big]^{2/q} dt \right)^{p/2}
\\
& \approx_{q} (1-s)^{-p/2} \left( \int_{0}^{\infty} r^{q(2-s)} \vert \partial^{2}_{r} P_{r} * f(x) \vert^{q} dr/r \right)^{p/2}.
\end{split}
\end{equation*}
And we can conclude using Lemma \ref{DTC}.

 \end{proof}
Lastly, we prove the upper bound regarding $ \dot{F}^{s}_{p,2}$. 
\begin{prop}\label{geq3}
Let $ q \in [2,\infty)$ then we have 
$$
  \int_{\mathbb{R}^{n}} \left( \int_{0}^{\infty} r^{q(1-s)-1} \big[ \int_{r}^{\infty} t \vert \partial^{2}_{t} P_{t} * f (x) \vert^{2} dt \big]^{q/2} dr \right)^{p/q} dx 
 \lesssim_{n,p,q}  {(1-s)^{-p/q}} [ f ]_{F^{s}_{p,2}}
$$

\end{prop}
\begin{proof}
Set $ F(x,t) : = \partial^{2}_{t} P_{t} * f (x)$. Then using the same argument in \eqref{EQS1} we get
\begin{equation}\label{EQS2}
\begin{split}
&   \int_{0}^{\infty} r^{q(1-s)-1} \big[ \int_{r}^{\infty} t \vert F(x,t) \vert^{2} dt \big]^{q/2} dr 
\\
&
=\int_{0}^{\infty} r^{-qs-1} \big[ \int_{1}^{\infty} \frac{1}{t^{2}}  \vert (rt)^{2} F(x,rt) \vert^{2} \frac{dt}{t} \big]^{q/2} dr
\\
& = \int_{0}^{\infty}  \big[ \int_{1}^{\infty} \frac{1}{(r^{s}t)^{2}}  \vert (rt)^{2} F(x,rt) \vert^{2} \frac{dt}{t}\big]^{q/2} \frac{dr}{r}.
\end{split}
\end{equation}
Note that since $ s \in (0,1)$ we have
\begin{equation*}
\begin{split}
& \int_{1}^{\infty} \frac{1}{(r^{s}t)^{2}}  \vert (rt)^{2} F(x,rt) \vert^{2} dt/t 
\\
&
{\leq} \int_{1}^{\infty} \frac{1}{(rt)^{2s}}  \vert (rt)^{2} F(x,rt) \vert^{2} dt/t 
\\
& 
\leq
\int_{0}^{\infty} t^{2(2-s)} \vert F(x,t) \vert^{2} dt/t.
\end{split}
\end{equation*}
Using this and \eqref{EQS2} we obtain 
\begin{equation}\label{EQS3}
\begin{split}
& \int_{0}^{\infty} r^{q(1-s)-1} \left( \int_{r}^{\infty} t \vert F(x,t) \vert^{2} dt \right)^{q/2} dr
\\
& 
= \int_{0}^{\infty}  \left( \int_{1}^{\infty} \frac{1}{(r^{s}t)^{2}}  \vert (rt)^{2} F(x,rt) \vert^{2} \frac{dt}{t} \right)^{q/2} \frac{dr}{r}
\\
& = \int_{0}^{\infty} \int_{1}^{\infty} \frac{1}{(r^{s}t)^{2}}  \vert (rt)^{2} F(x,rt) \vert^{2} \frac{dt}{t}   \left(
 \int_{1}^{\infty} \frac{1}{(r^{s}t)^{2}}  \vert (rt)^{2} F(x,rt) \vert^{2} \frac{dt}{t} \right)^{\frac{q}{2} -1} \frac{dr}{r}
 \\
 & 
 \leq  \int_{0}^{\infty} \int_{1}^{\infty} \frac{1}{(r^{s}t)^{2}}  \vert (rt)^{2} F(x,rt) \vert^{2} \frac{dt}{t}  \frac{dr}{r} \left( \int_{0}^{\infty} t^{2(2-s)} \vert F(x,t) \vert^{2} dt/t
  \right)^{\frac{q}{2} -1}.
\end{split}
\end{equation}
And note that we have
\begin{equation*}
\begin{split}
& \int_{0}^{\infty} \int_{1}^{\infty} \frac{1}{(r^{s}t)^{2}}  \vert (rt)^{2} F(x,rt) \vert^{2} \frac{dt}{t}  \frac{dr}{r}
\\
& = \int_{0}^{\infty} \int_{1}^{\infty} \frac{t^{2(s-1)}}{(rt)^{2s}}  \vert (rt)^{2} F(x,rt) \vert^{2} \frac{dt}{t}  \frac{dr}{r}
\\
& 
\\
& = \int_{1}^{\infty} t^{2(s-1)-1} \int_{0}^{\infty} (rt)^{-2s}\vert (rt)^{2} F(x,rt) \vert^{2} \frac{dr}{r} \frac{dt}{t}
\\
& = \int_{1}^{\infty} t^{2(s-1)-1} \int_{0}^{\infty} r^{-2s}\vert (r)^{2} F(x,r) \vert^{2} \frac{dr}{r} \frac{dt}{t}
\\
&
\lesssim (1-s)^{-1} \int_{0}^{\infty} r^{2(2-s)}  \vert \partial^{2}_{r} P_{r} * f(x) \vert^{2} dr/r.
\end{split}
\end{equation*}
Using this and \eqref{EQS3} we obtain 
\begin{equation*}
\begin{split}
& \int_{0}^{\infty} r^{q(1-s)-1} \left( \int_{r}^{\infty} t \vert \partial^{2}_{t} P_{t} * f (x) \vert^{2} dt \right)^{q/2} dr
\\
& \lesssim (1-s)^{-1} \left( \int_{0}^{\infty} t^{2(2-s)}  \vert \partial^{2}_{t} P_{t} * f(x) \vert^{2} dt/t \right)^{q/2}.
\end{split}
\end{equation*}
Raising both sides to power $ p/q $ and integrating in $ x $ gives the result upon using Lemma \ref{DTC}.
\end{proof}
Now Propisition \ref{U} follows immediately from Proposition \ref{geq1}, Proposition \ref{geq2}, and Proposition \ref{geq3}.

\subsection{Upper Bounds II}
In this section we prove the Sobolev estimates in Theorem \ref{SUU}. We do so as follows; By Proposition \ref{M} it is enough to estimate 
$$
\int_{\mathbb{R}^{n}} \left( \int_{0}^{\infty} r^{q(1-s)-1} \big[ \int_{r}^{\infty} t \vert \partial^{2}_{t} P_{t} * f (x) \vert^{2} dt \big]^{q/2} dr \right)^{p/q} dx.
$$
And note if we partition the middle integral for when $2^{-k} \leq  r \leq 2^{-k +1 }$  and partition the inner integral for when $2^{-j} \leq t \leq 2^{-j+1}$ and use Lemma \ref{Psn} then we get 

\begin{equation}\label{EQS4}
\begin{split}
& \int_{\mathbb{R}^{n}} \left( \int_{0}^{\infty} r^{q(1-s)-1} \big[ \int_{r}^{\infty} t \vert \partial^{2}_{t} P_{t} * f (x) \vert^{2} dt \big]^{q/2} dr \right)^{p/q} dx
\\
& 
\lesssim  \int_{\mathbb{R}^{n}} \left( \sum_{k \in \mathbb{Z}} 2^{qk(s-1)} \big[ \sum_{j < k} 2^{2j} \vert \mathcal{M}(f_{j}) \vert^{2} \big]^{q/2} \right)^{p/q},
\end{split}
\end{equation}

where $ f_{j}(x) = 2^{-2j} \partial^{2}_{t} P_{2^{-j}} * f(x) $. We prove Theorem \ref{SUU} by estimating the right hand side of \eqref{EQS4}. We split the sum when $ k > 0$ and when $k \leq 0$ and estimate each term separately.
\begin{prop}
Let $ p,q \in (1,\infty)$ and $ s \in (0,1)$ and $ \sigma \in (s,1)$. Then we have 
$$
\int_{\mathbb{R}^{n}} \left( \sum_{k >0} 2^{qk(s-1)} \big[ \sum_{j < k} 2^{2j} \vert \mathcal{M}(f_{j})(x) \vert^{2} \big]^{q/2} \right)^{p/q} dx  \lesssim_{p,q,n} (\sigma -s )^{-p/q} [f]^{p}_{\dot{F}^{\sigma}_{p,2}}
$$

\end{prop}
\begin{proof}
\begin{equation*}
\begin{split}
& \int_{\mathbb{R}^{n}} \left( \sum_{k >0} 2^{qk(s-1)} \big[ \sum_{j < k} 2^{2j} \vert \mathcal{M}(f_{j})(x) \vert^{2} \big]^{q/2} \right)^{p/q} dx 
\\
& = \int_{\mathbb{R}^{n}} \left( \sum_{k >0} 2^{qks} \big[ \sum_{j < k} \vert 2^{j-k}  \mathcal{M}(f_{j})(x) \vert^{2} \big]^{q/2} \right)^{p/q} dx.
\end{split}
\end{equation*}
For the inner sum, set $l =j-k$ then the above becomes
\begin{equation*}
\begin{split}
& = \int_{\mathbb{R}^{n}} \left( \sum_{k >0} 2^{qks} \big[ \sum_{l < 0} \vert 2^{l}  \mathcal{M}(f_{l+k}) \vert^{2} \big]^{q/2} \right)^{p/q} dx  
\\
& \leq \int_{\mathbb{R}^{n}} \left( \sum_{k >0} 2^{qks} \big[ \sum_{l < 0} \vert 2^{l\sigma} \mathcal{M} (f_{l+k})(x) \vert^{2} \big]^{q/2} \right)^{p/q}  dx
\\
& 
= \int_{\mathbb{R}^{n}} \left( \sum_{k >0} 2^{qk(-\sigma +s)} \big[ \sum_{l < 0} \vert 2^{l\sigma} 2^{k \sigma}  \mathcal{M}(f_{l+k}) \vert^{2} \big]^{q/2} \right)^{p/q} dx  
\\
& \leq \int_{\mathbb{R}^{n}} \left( \sum_{k > 0} 2^{q(-\sigma +s)} \big[ \sum_{l \in \mathbb{Z}} 2^{2 l  \sigma} |\mathcal{M}(f_{l})(x)|^{2} \big]^{q/2} \right)^{p/q}
\\
& 
\overset{\eqref{elementaryest1}}{\lesssim_{n,p,q}}  \frac{ 1}{ (\sigma - s)^{p/q}} \int_{\mathbb{R}^{n}}\big[ \sum_{l \in \mathbb{Z}} \vert 2^{l\sigma}   \mathcal{M}(f_{l}) \vert^{2} \big]^{p/2}dx
\\
& \overset{\text{Lemma } \ref{SteinF}}{\lesssim_{n,p,q}} (\sigma -s)^{-p/q} [f]_{\dot{F}^{\sigma}_{p,2}}^{p}.
\end{split}
\end{equation*}
\end{proof}
Next, the lower derivative will be obtained when we sum over negative $ k$. 
\begin{prop}
Let $ p,q \in (1,\infty)$ and $ s \in (0,1)$ and $ v \in (0,s)$ then we have
$$
\int_{\mathbb{R}^{n}} \left( \sum_{k \leq 0} 2^{qk(s-1)} \big[ \sum_{j < k} 2^{2j} \vert \mathcal{M}(f_{j})(x) \vert^{2} \big]^{q/2} \right)^{p/q} dx \lesssim_{n,p,q} (s-v)^{-p/q} [f]_{\dot{F}^{v}_{p,2}}^{p}.
$$
\end{prop}
\begin{proof}
Upon changing the index of summation, as in the previous proof, we have
\begin{equation*}
\begin{split}
& \int_{\mathbb{R}^{n}} \left( \sum_{k \leq 0} 2^{qks} \big[ \sum_{j < 0} \vert 2^{j}  \mathcal{M}(f_{j+k})(x) \vert^{2} \big]^{q/2} \right)^{p/q}dx 
\\
& \leq \int_{\mathbb{R}^{n}} \left( \sum_{k \leq 0} 2^{qks} \big[ \sum_{j < 0} \vert 2^{vj} \mathcal{M}( f_{j+k})(x) \vert^{2} \big]^{q/2} \right)^{p/q} 
\\
& = \int_{\mathbb{R}^{n}} \left( \sum_{k \leq 0} 2^{qk(s-v)} \big[ \sum_{j < 0} \vert 2^{vj+vk}  \mathcal{M}(f_{j+k})(x) \vert^{2} \big]^{q/2} \right)^{p/q} 
\\
& \leq \int_{\mathbb{R}^{n}} \left( \sum_{k \leq 0} 2^{qk(s-v)} \big[ \sum_{j \in \mathbb{Z}} \vert 2^{vj}  \mathcal{M}(f_{j})(x) \vert^{2} \big]^{q/2} \right)^{p/q} 
\\
& \overset{\eqref{elementaryest1}}{\lesssim}  (s-v)^{-p/q} \int_{\mathbb{R}^{n}} \left( \sum_{j \in \mathbb{Z}} \vert 2^{j v} \mathcal{M}(f_{j})(x) \vert^{2} \right)^{p/2} dx
\\
& \overset{\text{Lemma } \ref{SteinF}}{\lesssim}(s-v)^{-p/q} [f]_{\dot{F}^{v}_{p,2}}.
\end{split}
\end{equation*}
\end{proof}
Now the proof of Theorem \ref{SUU} is a consequence of the previous two Propositions and Proposition \ref{M} and  \eqref{EQS4}.

\subsection{Lower bounds I}
Here we prove the lower bounds in Theorem \ref{q12} and \ref{QA1}
\begin{prop}\label{L}
let $ p ,q \in (1,\infty)$ then we have 
\begin{enumerate}
    \item if $ q \in (1,2]$ then 
    \begin{equation}\label{lowerboundforE1}
    [f]_{E^{s}_{p,q}} \gtrsim_{n,p,q} (1-s)^{-1/2} [f]_{F^{s}_{p,q}}.
    \end{equation}
    \item  If $ q \in [2, \infty)$ then 
\begin{equation}\label{lowerboundforE2}
    [f]_{E^{s}_{p,q}} \gtrsim_{n,p,q} (1-s)^{-1/q} [f]_{F^{s}_{p,q}}.
    \end{equation}
    \item  If $ q \in (1,2]$ then
    \begin{equation}\label{lowerboundforE3}
     [f]_{E^{s}_{p,q}} \gtrsim_{n,p,q}  (1-s)^{-1/q} [f ]_{F^{s}_{p,2}}.
     \end{equation}
\end{enumerate}
\end{prop}
We will use a duality argument to prove the above proposition. However, we would like to mention that one can use the lower bound proven in Proposition \ref{M} to get a more direct proof for Proposition \ref{L}. However, we prefer the duality argument since it is what gives the lower bounds for the Sobolev estimates in Theorem \ref{LowerS}, which we will use to conclude a global version of BBM2 theorem for $E^{s}_{p,q}.$

\begin{prop}\label{DualityE}
Fix $ q, p \in (1,\infty) $ and $ s \in (0,1)$. Then for $ f \in \SSS(\R^{n})$,  we have  
$$
[f]_{\dot{F}^{s}_{p,q}} \lesssim_{n,p,q} \sup_{g} (1-s) [f]_{E^{s}_{p,q}}\mbox{  } [g]_{E^{s}_{p',q'}}
$$
where the sup is taken over all $g \in \SSS(\R^{n}) $ with $[g]_{\dot{F}^{s}_{p',q'}} \leq 1$. 
\end{prop}
\begin{proof}
By Lemma \ref{2.10} we have
\begin{equation}\label{mod1}
\begin{split}
& [f]_{\dot{F}^{s}_{p,q}} 
\\
& \lesssim \sup_{[g]_{\dot{F}^{s}_{p,q}} \leq 1} \vert \int_{\R^{n}} (-\Delta)^{s/2} (f)(x) (-\Delta)^{s/2} (g)(x) dx \vert.
\end{split}
\end{equation}
Fix any $g$ with the indicated constraints, then we simplify the right hand side of \eqref{mod1}. By Plancheral, 
\begin{equation}\label{mod2}
\begin{split}
& \vert \int_{\R^{n}} (-\Delta)^{s/2} (f)(x) (-\Delta)^{s/2} (g)(x) dx \vert
\\
& = C_{n} \vert \int_{\R^{n}} \vert \xi \vert^{s} \hat{f}(\xi) \vert \xi \vert^{s} \hat{g}(\xi) d\xi \vert
\\
& = C_{n} \vert \int_{\R^{n}} \vert \xi \vert^{2s} \hat{f}(\xi) \hat{g}(\xi) d\xi \vert.
\end{split}
\end{equation}
Next, notice the following identity; for any $ \alpha > 0 $, $ s \in (0,1)$ we have 
\begin{equation}\label{mod3}
\begin{split}
    & \int_{0}^{\infty} t^{2(1-s)-1} \alpha^{2} e^{- 2t \alpha} dt = 2^{-2(1-s)} \alpha^{2s} \Gamma(2(1-s)) \,
\end{split}
\end{equation}
where $\Gamma$ is the gamma function. Indeed, \eqref{mod3} follows by a change of variable and the definition of the Gamma function
\begin{equation*}
\begin{split}
& \int_{0}^{\infty} t^{2(1-s)-1} \alpha^{2} e^{-2t \alpha} dt
\\
& = \alpha^{2} \int_{0}^{\infty} t^{2(1-s)-1} e^{-2t \alpha } dt
\\
& \overset{r := 2 \alpha t}{=} \alpha^{2} \int_{0}^{\infty} (\frac{r}{2\alpha})^{2(1-s)} e^{ - r } \frac{dr}{r}
\\
& =  2^{-2 (1-s)} \alpha^{2} \alpha^{-2(1-s)} \int_{0}^{\infty} r^{2(1-s)-1 } e^{-r} dr
\\
& = 2^{-2(1-s)}\alpha^{2s} \Gamma(2(1-s)).
\end{split}
\end{equation*}
This proves \eqref{mod3}. And so we find 
\begin{equation}\label{mod4}
\begin{split}
& \int_{0}^{\infty} t^{2(1-s)-1} \int_{\R^{n}} (\partial_{t} P_{t} * f(x) ) ( \partial_{t} P_{t} * g(x) ) dx
\\
& \overset{\text{Plancheral}}{=} \int_{0}^{\infty } t^{2(1-s)-1} \int_{\R^{n}} |\xi | e^{-t \vert \xi \vert} \hat{f}(\xi) \vert \xi \vert e^{-t \vert \xi \vert} \hat{g}(\xi) d\xi
\\
& = \int_{\R^{n}} \hat{f}(\xi) \hat{g}(\xi) \int_{0}^{\infty} t^{2(1-s)-1} \vert \xi \vert^{2} e^{-2t\vert \xi \vert } d\xi
\\
& \overset{\eqref{mod3}}{=} 2^{-2(1-s)} \Gamma(2(1-s))\int_{\R^{n}} \vert \xi \vert^{2s} \hat{f}(\xi) \hat{g}(\xi) d\xi.
\end{split}
\end{equation}
Combining  \eqref{mod2} and \eqref{mod4} we find
\begin{equation*}
\begin{split}
& \vert \int_{\R^{n}} (-\Delta)^{s/2} (f)(x) (-\Delta)^{s/2} (g)(x) dx \vert
\\
& \lesssim \Gamma(2(1-s))^{-1} \vert \int_{0}^{\infty} t^{2(1-s) -1} \int_{\R^{n}} (\partial_{t} P_{t} * f(x) ) ( \partial_{t} P_{t} * g(x) ) dx \vert
\\
& \overset{\text{H\"older}}{\lesssim} \Gamma(2(1-s))^{-1} [f]_{E^{s}_{p,q}} [g]_{E^{s}_{p',q'}}
\\
& \overset{\eqref{gammafunctionbound}}{\lesssim} (1-s) [f]_{E^{s}_{p,q}} [g]_{E^{s}_{p',q'}}.
\end{split}
\end{equation*}
Taking the supremum over $g \in \SSS(\mathbb{R}^{n})$ gives the result by using \eqref{mod1}.
\end{proof}

The lower bounds in Proposition \ref{L} are a consequence of Proposition \ref{DualityE} above.
\begin{proof}[ Proof of Proposition \ref{L}]
By Proposition \ref{DualityE}  
$$
 [f]_{\dot{F}^{s}_{p,q}} \lesssim \sup_{g} (1-s) [f]_{E^{s}_{p,q}} [g]_{E^{s}_{p',q'}}
$$
where the supremum is taken over $g \in \SSS(\mathbb{R}^{n})$ with $ [g]_{\dot{F}^{s}_{p',q'}} \leq 1$. Fix such a $g$. Then it is enough to estimate 
$$
(1-s) [f]_{E^{s}_{p,q}} [g]_{E^{s}_{p',q'}}.
$$
First we prove (1), i.e \eqref{lowerboundforE1}. To that end, let $ q \in (1,2] $ then $ q' \in [2,\infty)$. Then by using \eqref{upperboundforE2} we find
\begin{equation*}
\begin{split}
[g]_{E^{s}_{p',q'}} \lesssim_{n,p,q} (1-s)^{-1/2} [g]_{\dot{F}^{s}_{p',q'}} \lesssim (1-s)^{-1/2}.
\end{split}
\end{equation*}
Thus, 
\begin{equation*}
\begin{split}
& (1-s) [f]_{E^{s}_{p,q}} [g]_{E^{s}_{p',q'}} \lesssim_{n,p,q} (1-s)^{1/2} [f]_{E^{s}_{p,q}}.
\end{split}
\end{equation*}
This proves \eqref{lowerboundforE1}.\\
Next, we prove part (2), i.e \eqref{lowerboundforE2}. Let $ q \in [2,\infty)$ then $ q' \in (1,2]$ and we find by \eqref{upperboundforE1}
\begin{equation*}
\begin{split}
[g]_{E^{s}_{p',q'}} \lesssim_{n,p,q} (1-s)^{-1/q} [g]_{\dot{F}^{s}_{p',q'}} \lesssim (1-s)^{-1/q}.
\end{split}
\end{equation*}
Therefore,
\begin{equation*}
\begin{split}
& (1-s) [f]_{E^{s}_{p,q}} [g]_{E^{s}_{p',q'}} \lesssim_{n,p,q} (1-s)^{1- \frac{1}{q'}} [f]_{E^{s}_{p,q}} = (1-s)^{1/q} [f]_{E^{s}_{p,q}}.
\end{split}
\end{equation*}
This proves \eqref{lowerboundforE2}. Lastly, we prove part (3) of the proposition, \eqref{lowerboundforE3}. To that end, let $ q \in (1,2]$ then $ q' \in [2,\infty)$ and so we find by \eqref{upperboundforE3}
\begin{equation*}
\begin{split}
& [g]_{E^{s}_{p',q'}} \lesssim_{n,p,q} (1-s)^{-1/q'} [g]_{\dot{F}^{s}_{p',q'}} \lesssim (1-s)^{-1/q'}.
\end{split}
\end{equation*}
Hence, 
\begin{equation*}
\begin{split}
& (1-s) [f]_{E^{s}_{p,q}} [g]_{E^{s}_{p',q'}}
\\
& \lesssim_{n,p,q} (1-s)^{1/q} [f]_{E^{s}_{p,q}}.
\end{split}
\end{equation*}
This finishes the proof of \eqref{lowerboundforE3} and concludes the proof of the proposition.

\end{proof}
\subsection{Lower bounds II}
In this subsection we prove Theorem \ref{LowerS}. We again use a duality argument but we modify it to get more derivatives hitting $f$. More precisely, we wish to have on the left hand side $ [f]_{\dot{F}^{\sigma}_{p,2}}$ while on the right hand side $ [f]_{E^{s}_{p,q}}$ where $ s \geq \sigma$. This is the content of the next proposition
\begin{prop}\label{Duality2}
Let $ p,q \in (1,\infty)$ and $ s \in (0,1)$ and $ \sigma \in (0,1)$ then we have the following
\begin{equation*}
\begin{split}
&[f]_{\dot{F}^{\sigma}_{p,2}} 
\lesssim_{p,q,n} \Vert f \Vert_{L^{p}} 
\\
&+(1-\sigma)[f]_{E^{s}_{p,q}} \mbox{ } \sup_{g}\mbox{  } \left( \int_{\mathbb{R^{n}}}\left( \int_{0}^{\infty} r^{q'(1-(2 \sigma -s))} \left( \int_{r}^{\infty} \vert \frac{g(x,t)}{t} \vert^{2} \frac{dt}{t} \right)^{q'/2} \frac{dr}{r} \right)^{p'/q'} dx \right)^{1/p'}.
\end{split}
\end{equation*}
Where the supremum is taken over all $ g \in L^{1}_{loc}(\mathbb{R}^{n} \times (0,\infty))$ so that $ g(x,t) = 0 $ if $ t > 1$ and 
$$
\int_{\mathbb{R}^{n}} \left( \int_{0}^{\infty} \vert t^{-\sigma} g(x,t) \vert^{2} \frac{dt}{t} \right)^{p'/2} dx \leq 1.
$$
\end{prop}
\begin{proof}
 We have by Lemma \ref{DTC} 
\begin{equation}\label{moodd}
\begin{split}
&[f]_{\dot{F}^{\sigma}_{p,2}} \approx_{n,p,q} \left( \int_{\mathbb{R}^{n}} \left( \int_{0}^{\infty} t^{2 (2-\sigma)} \vert \partial^{2}_{t} P_{t} *f (x) \vert^{2} dt/t \right)^{p/2}  dx \right)^{1/p} 
\\
&  \lesssim \left( \int_{\mathbb{R}^{n}} \left( \int_{1}^{\infty} t^{2 (2-\sigma)} \vert \partial^{2}_{t} P_{t} *f (x) \vert^{2} dt/t \right)^{p/2}  dx \right)^{1/p}  
\\
& + \left( \int_{\mathbb{R}^{n}} \left( \int_{0}^{1} t^{2 (2-\sigma)} \vert \partial^{2}_{t} P_{t} *f (x) \vert^{2} dt/t \right)^{p/2}  dx \right)^{1/p}. 
\end{split}
\end{equation}
For the first term on the right hand side of \eqref{moodd} we use Lemma \ref{LP} to find
\begin{equation*}
\begin{split}
& \left( \int_{\mathbb{R}^{n}} \left( \int_{1}^{\infty} t^{2 (2-\sigma)} \vert \partial^{2}_{t} P_{t} *f (x) \vert^{2} dt/t \right)^{p/2}  dx \right)^{1/p} 
\\
& \leq \left( \int_{\mathbb{R}^{n}} \left( \int_{1}^{\infty} t^{4} \vert \partial^{2}_{t} P_{t} *f (x) \vert^{2} dt/t \right)^{p/2}  dx \right)^{1/p}
\\
& \leq \left( \int_{\mathbb{R}^{n}} \left( \int_{0}^{\infty} t^{4} \vert \partial^{2}_{t} P_{t} *f (x) \vert^{2} dt/t \right)^{p/2}  dx \right)^{1/p}
\\
& \overset{\text{Lemma } \ref{LP}}{ \lesssim}\Vert f \Vert_{L^{p}}.
\end{split}
\end{equation*}
We plug this back into \eqref{moodd} to obtain
\begin{equation}\label{moodd2}
\begin{split}
& [f]_{\dot{F}^{\sigma}_{p,2}}
\\
& \lesssim \Vert f \Vert_{L^{p}} + \left( \int_{\mathbb{R}^{n}} \left( \int_{0}^{1} t^{2 (2-\sigma)} \vert \partial^{2}_{t} P_{t} *f (x) \vert^{2} dt/t \right)^{p/2}  dx \right)^{1/p}.
\end{split}
\end{equation}
Next, we need to estimate the second term on the right hand side of \eqref{moodd}. 
Put $ F(x,t) = t^{2} \partial^{2}_{t} P_{t} * f(x)$. Then we think of $ t^{-\sigma} F $ as an element of the Banach $ L^{p,2} $ defined by 
$$
L^{p,2} = \{ G \in L^{1}( \mathbb{R}^{n} \times (0,1) )) : \int_{\mathbb{R}^{n}} \left( \int_{0}^{1} \vert G(x,t ) \vert^{2} \frac{dt}{t} \right)^{p/2} dx < \infty \}.
$$
Then, by duality we obtain
\begin{equation*}
\begin{split}
& 
\left( \int_{\mathbb{R}^{n}} \left( \int_{0}^{1} t^{2 (2-\sigma)} \vert \partial^{2}_{t} P_{t} *f (x) \vert^{2} dt/t \right)^{p/2}  dx \right)^{1/p}
\\
& = \sup_{h} \int_{\mathbb{R}^{n}}  \int_{0}^{1} t^{-\sigma} F(x,t) h(x,t) \frac{dt}{t},
\end{split}
\end{equation*}
where the sup is taken over $ h \in L^{p',2}$ with norm $1$. Fix such $ h $ and put $ g(x,t)= t^{\sigma} h(x,t) $. And extend $ g $ to be $ 0$ when $ t > 1 $. Then we have
\begin{equation*}
\begin{split}
&\int_{\mathbb{R}^{n}}  \int_{0}^{1} t^{-\sigma} F(x,t) h(x,t) \frac{dt}{t} 
\\
& 
= \int_{\mathbb{R}^{n}}  \int_{0}^{\infty} t^{-2\sigma} F(x,t) g(x,t) \frac{dt}{t} 
\\
&
=
2(1-\sigma) \int_{\mathbb{R}^{n}} \int_{0}^{\infty} F(x,t) g(x,t) t^{-2} \left( \int_{0}^{t} r^{2(1-\sigma)} \frac{dr}{r} \right) \frac{dt}{t}
\\
&
\approx (1- \sigma) \int_{\mathbb{R}^{n}} \int_{0}^{\infty} r^{2(1-\sigma)} \left(\int_{r}^{\infty} F(x,t) g(x,t) t^{-2} \frac{dt}{t} \right) \frac{dr}{r}
\\
& = (1-\sigma) \int_{\mathbb{R}^{n}} \int_{0}^{\infty} r^{1-s} r^{1 - (2\sigma-s) }\left(\int_{r}^{\infty} F(x,t) g(x,t) t^{-2} \frac{dt}{t} \right) \frac{dr}{r}.
\end{split}
\end{equation*}
Now applying H\"{o}lder's inequality three times and using Proposition \ref{M} we obtain
\begin{equation*}
\begin{split}
&\int_{\mathbb{R}^{n}}  \int_{0}^{1} t^{-\sigma} F(x,t) h(x,t) \frac{dt}{t} 
\\
&
\lesssim_{n,p,q} (1-\sigma) [f]_{E^{s}_{p,q}} \left( \int_{\mathbb{R}^{n}}\left( \int_{0}^{\infty} r^{q'(1-(2 \sigma -s))} \left( \int_{r}^{\infty} \vert \frac{g(x,t)}{t} \vert^{2} \frac{dt}{t} \right)^{q'/2} \frac{dr}{r} \right)^{p'/q'} dx \right)^{1/p'}.
\end{split}
\end{equation*}
Then taking the supremum over $ h $, and hence over $g$, with norm $ 1 $ give us 
\begin{equation*}
\begin{split}
&  \left( \int_{\mathbb{R}^{n}} \left( \int_{0}^{1} t^{2 (2-\sigma)} \vert \partial^{2}_{t} P_{t} *f (x) \vert^{2} dt/t \right)^{p/2}  dx \right)^{1/p}
\\
& \lesssim \sup_{g} (1- \sigma) [f]_{E^{s}_{p,q}} \left( \int_{\mathbb{R}^{n}}\left( \int_{0}^{\infty} r^{q'(1-(2 \sigma -s))} \left( \int_{r}^{\infty} \vert \frac{g(x,t)}{t} \vert^{2} \frac{dt}{t} \right)^{q'/2} \frac{dr}{r} \right)^{p'/q'} dx \right)^{1/p'},
\end{split}
\end{equation*}
where the supremum is taken over $g$ with the indicated constraints. Plugging this back into \eqref{moodd} proves the result.

\end{proof}

With the above we can proceed to prove Theorem \ref{LowerS}.
\begin{proof}[Proof of Theorem \ref{LowerS}]
Let $ \Theta> 1 $ and $ s \in (1-\frac{1}{2 \Theta},1)$ and $ \bar{\sigma} \in (0,s)$ with $ (1-\bar{\sigma}) = \Theta (1-s)$. In particular, $ \sigma \in (s/2,s)$. It suffices to prove the theorem for $ \bar{\sigma}$. Indeed, since we have the obvious embedding
$$
[f]_{\dot{F}^{\sigma}_{p,2}} \lesssim \Vert f \Vert_{L^{p}} + [f]_{\dot{F}^{\bar{\sigma}}_{{p,2}}}.
$$
Therefore, we prove the theorem for $ \bar{\sigma}$.
Using Proposition \ref{Duality2} we get
\begin{equation}\label{EQ1}
\begin{split}
& [f]_{\dot{F}^{\bar{\sigma}}_{p,2}}
\lesssim_{n,p,q} \Vert f \Vert_{L^{p}}
\\
&+(1-\bar{\sigma})[f]_{E^{s}_{p,q}} \sup_{g}  
\left( \int_{\mathbb{R}^{n}}\left( \int_{0}^{\infty} r^{q'(1-(2 \bar{\sigma} -s))} \left( \int_{r}^{\infty} \vert \frac{g(x,t)}{t} \vert^{2} \frac{dt}{t} \right)^{q'/2} \frac{dr}{r} \right)^{p'/q'} dx \right)^{1/p'}.
\end{split}
\end{equation}
Where $ g$ satisfies the assumptions written in Proposition \ref{Duality2}. 
Fix any such $g$, then if we prove
\begin{equation}\label{EQF}
\begin{split}
&\left( \int_{\mathbb{R^{n}}}\left( \int_{0}^{\infty} r^{q'(1-(2 \bar{\sigma} -s))} \left( \int_{r}^{\infty} \vert \frac{g(x,t)}{t} \vert^{2} \frac{dt}{t} \right)^{q'/2} \frac{dr}{r} \right)^{p'/q'} dx \right)^{1/p'}
\\
& \lesssim_{n,p,q} (s-\bar{\sigma})^{-1/q'}.
\end{split}
\end{equation}
Then the result follows. Indeed if we prove \eqref{EQF} then using it together with \eqref{EQ1} we get  
$$
[f]_{\dot{F}^{s}_{p,q}} \lesssim_{n,p,q} \Vert f \Vert_{L^{p}} + \frac{(1-\bar{\sigma})}{(s-\bar{\sigma})^{1/q'}} [f]_{E^{s}_{p,q}}.
$$
And the result follows since by assumption $ s - \bar{\sigma} = ( \Theta-1)(1-s)$. The proof of \eqref{EQF} is similar to that of Proposition 3.7. First note that by the support properties of $ g $ we have 
\begin{equation*}
\begin{split}
& \int_{0}^{\infty} r^{q'( 1- ( 2 \bar{\sigma} -s ))} \left( \int_{r}^{\infty} \vert \frac{g(x,t)}{t} \vert^{2} \frac{dt}{t} \right)^{q'/2} \frac{dr}{r} 
\\
& =\int_{0}^{1} r^{q'( 1- ( 2 \bar{\sigma} -s ))} \left( \int_{r}^{\infty} \vert \frac{g(x,t)}{t} \vert^{2} \frac{dt}{t} \right)^{q'/2} \frac{dr}{r}.
\end{split}
\end{equation*}
Next, note that by the choice of $ \bar{\sigma}$ we have that $ 0 < 2 \bar{\sigma} -s < \bar{\sigma}$.
Therefore,
\begin{equation*}
\begin{split}
&\int_{0}^{1} r^{q'( 1- ( 2 \bar{\sigma} -s ))} \left( \int_{r}^{\infty} \vert \frac{g(x,t)}{t} \vert^{2} \frac{dt}{t} \right)^{q'/2} \frac{dr}{r}
\\
& = \int_{0}^{1} r^{- q'( 2 \bar{\sigma} -s)} \left( \int_{1}^{\infty} \vert \frac{g(x,rt)}{t} \vert^{2} \frac{dt}{t} \right)^{q'/2} \frac{dr}{r}
\\
&
\leq \int_{0}^{1} r^{- q'( \bar{\sigma} -s)} \left( \int_{1}^{\infty} \vert g(x,rt)(rt)^{-\bar{\sigma}} \vert^{2} \frac{dt}{t} \right)^{q'/2} \frac{dr}{r} 
\\
&
\leq \int_{0}^{1} r^{- q'(  \bar{\sigma} -s)} \left( \int_{0}^{\infty} \vert g(x,t) t^{-\bar{\sigma}} \vert^{2} \frac{dt}{t} \right)^{q'/2} \frac{dr}{r} 
\\
& \approx_{q} (s- \bar{\sigma})^{-1} \left( \int_{0}^{\infty} \vert g(x,t) t^{-\bar{\sigma}} \vert^{2} \frac{dt}{t} \right)^{q'/2}.
\end{split}
\end{equation*}
Now \eqref{EQF} follows once we raise to power $ p'/q' $ and integrate in $x$ and use the fact that $ t^{-\sigma} g $ has norm $ 1 $ in $ L^{p',2}$.

\end{proof}
\subsection{BBM type theorems} In this subsection we apply the inequalities in the previous two subsections to conclude  two BBM theorems for the semi-norm $ E^{s}_{p,q}$. 

\begin{proof}[Proof of Theorem \ref{BBM1PR} ]
First take $ s \in (1/2,1)$ and put $ v = 2s-1$. Then clearly $ v \in (0,s)$. Then using Theorem \ref{SUU} with $ \sigma = 1$ and $ v = 2s-1$ we get 
$$
[f]_{E^{s}_{p,q}} \lesssim_{n,p,q} (1 -s)^{-1/q} [f]_{\dot{F}^{1}_{p,2}} + ( s- ( 2s-1))^{-1/q}[f]_{\dot{F}^{2s-1}_{p,2}}.
$$
Which translates to 
$$
(1-s)^{1/q} [f]_{E^{s}_{p,q}} \lesssim_{n,p,q} \Vert \nabla f \Vert_{L^{p}} + [f]_{\dot{F}^{2s-1}_{p,2}}.
$$
Letting $ s \to 1 $ gives us
$$
\lim_{s \to 1^{-}} (1-s)^{1/q}[f]_{E^{s}_{p,q}} \lesssim_{n,p,q} \Vert \nabla f \Vert_{L^{p}}.
$$
Now we need to prove the other inequality. If $ q \in (1,2]$ then from Theorem \ref{L} we get the other inequality. Namely we get 
$$
[f]_{\dot{F}^{s}_{p,2}} \lesssim_{n,p,q} (1-s)^{1/q}[f]_{E^{s}_{p,q}}.
$$
Letting $ s \to 1 $ gives 
$$
\Vert \nabla f \Vert_{L^{p}} \lesssim_{n,p,q} \lim_{s\to 1^{-}}(1-s)^{1/q} [f]_{E^{s}_{p,q}}.
$$
We prove the same for when $ q \in (2,\infty)$. Actually this argument works for any $ q \in (1,\infty)$. Take $ R > 0 $ arbitrary, then 
\begin{equation*}
\begin{split}
& (1-s)^{1/q} \left( \int_{\vert x \vert <R} \left( \int_{0}^{\infty} t^{q(1-s)-1} \vert \partial_{t} P_{t} * f (x) \vert^{q} dt \right)^{p/q} dx \right)^{1/p}
\\
& = \left( \int_{\vert x \vert < R} \left( (1-s) \int_{0}^{\infty} t^{q(1-s)-1} \vert \partial_{t} P_{t} * f (x) \vert^{q} dt \right)^{p/q} dx \right)^{1/p}
\\
& \geq \left( \int_{\vert x \vert < R} \left( (1-s) \int_{0}^{(1-s)} t^{q(1-s)-1} \vert \partial_{t} P_{t} * f (x) \vert^{q} dt \right)^{p/q} dx \right)^{1/p}
\\
& = \left( \int_{\vert x\vert < R } \left(  \frac{1}{q} \int_{0}^{(1-s)} \partial_{t} (t^{q(1-s)}) \vert \partial_{t} P_{t} * f (x) \vert^{q} dt \right)^{p/q} dx \right)^{1/p}.
\end{split}
\end{equation*}
Now we do integration by parts on the inner integral to obtain

 \begin{equation*}
\begin{split}
    & \int_{0}^{(1-s)} \partial_{t} (t^{q(1-s)}) \vert \partial_{t} P_{t} * f (x) \vert^{q} dt
    \\
    & = \vert (1-s)^{(1-s)} \partial_{t} P_{(1-s)} * f (x) \vert^{q} - \int_{0}^{(1-s)}  t^{q(1-s)} \partial_{t} | \partial_{t} P_{t} * f (x) |^{q} dt.   
\end{split}
 \end{equation*}
Define
$$
G_{s}(x) = \vert (1-s)^{(1-s)} \partial_{t} P_{(1-s)} * f (x) \vert^{q} - \int_{0}^{(1-s)}  t^{q(1-s)} \partial_{t} ( \partial_{t} P_{t} * f (x) )^{q} dt.  
$$
We claim that $ G_{s}(x) \to \vert \Delta^{1/2} f (x) \vert^{q}$, as $ s \to 1$. Indeed, by Lemma \ref{PC} we get that 
$$
\lim_{ s \to 1} \vert (1-s)^{(1-s)} \partial_{t} P_{(1-s)} * f (x) \vert^{q} = C \vert  \Delta^{1/2} f(x) \vert^{q}.
$$
We just need to show that 
$$
\int_{0}^{(1-s)}  t^{q(1-s)} \partial_{t} | \partial_{t} P_{t} * f (x) |^{q} dt \to 0
$$
as $ s \to 1 $. This is clear since the integrand is bounded, indeed
\begin{equation}\label{dominated1}
\begin{split}
& \vert \partial_{t} | \partial_{t} P_{t} * f (x) |^{q} \vert =\vert q ( \partial_{t} P_{t} * f (x) )^{q-1} \partial_{t}^{2} P_{t} * f(x) \vert 
\\
& \leq q \left( \int_{\mathbb{R}^{n}} \vert \xi \vert e^{- \vert \xi \vert t} \vert \hat{f} (\xi ) \vert d \xi \right)^{q-1} \left( \int_{\mathbb{R}^{n}} \vert \xi \vert^{2} e^{- \vert \xi \vert t} \vert \hat{f} (\xi ) \vert d \xi \right)
\\
& \leq C(q,f). 
\end{split}
\end{equation}
Therefore, 
$$
\int_{0}^{(1-s)}  t^{q(1-s)} \partial_{t} | \partial_{t} P_{t} * f (x) |^{q} dt \lesssim_{f,q} (1-s).
$$
Letting $ s \to 1$ shows that we have
$$
G_{s}(x) \to C \vert \Delta^{1/2} f(x) \vert^{q}.
$$
Moreover, notice that $ G_{s}$ is uniformly bounded. That is, there exists $C:= C(f,q,n)$ so that $ G_{s} (x) \leq C(f,q,n)$ for any $ s \in (0,1)$ and $x \in \mathbb{R}^{n}$.  Indeed, notice
\begin{equation}\label{dominated2}
\begin{split}
& | (1-s)^{1-s} \partial_{t} P_{1-s} *f (x) |^{q}
\\
& \lesssim | \partial_{t} P_{1-s} *f(x) |^{q} 
\\
& \leq C(f,q,n).
\end{split}
\end{equation}
Combining \eqref{dominated1} and \eqref{dominated2} and the definition of $G_{s}$ we obtain $ G_{s} (x) \leq C(f,q,n)$ for any $ s \in (0,1)$ and any $ x \in \mathbb{R}^{n}$. 
Hence, by Dominated convergence theorem we get 
$$
\int_{\vert x \vert < R} \vert G_{s}(x) \vert^{p/q} dx  \to \int_{\vert x \vert <R } \vert \Delta^{1/2} f (x) \vert^{p} dx,
$$
where the dominating function being $ \chi_{|x| \leq R} C(f,q,n).$
Thus, we proved the following 
\begin{equation*}
\begin{split}
& \lim_{s\to 1^{-}}(1-s)^{1/q}[f]_{E^{s}_{p,q}}
\\
& \geq
\lim_{s\to 1^{-}}(1-s)^{1/q} \left( \int_{\vert x \vert <R} \left( \int_{0}^{\infty} t^{q(1-s)-1} \vert \partial_{t} P_{t} * f (x) \vert^{q} dt \right)^{p/q} dx \right)^{1/p}
\\
& \gtrsim_{n,p,q} \left( \int_{\vert x \vert < R } \vert \Delta^{1/2} f(x) \vert^{p} dx \right)^{1/p}.
\end{split}
\end{equation*}
For any $ R> 0$. The result follows once we let $ R \to \infty$, and use the fact that 
$$
\Vert \Delta^{1/2} f \Vert_{L^{p}} \approx_{n,p} \Vert \nabla f \Vert_{L^{p}}.
$$

\end{proof}

Lastly, the last result we prove regarding the semi-norm $E^{s}_{p,q}$ is BBM2. 
 \begin{proof}[Proof of Corollary \ref{BBM2PR}]
This is an immediate consequence of Theorem \ref{LowerS}. Indeed, upon removing finitely many $k$'s we may assume that $ (1-s_{k}) < \frac{1}{4}$. Then using Theorem \ref{LowerS} with $ \Theta =2 $ we obtain
$$
[f_{k}]_{\dot{F}^{\sigma}_{p,2}} \lesssim_{n,p,q}  \Vert f_{k} \Vert_{L^{p}} + ( 1-s_{k} )^{1/q}[f_{k}]_{E^{s_{k}}_{p,q}}.
$$
For $ \sigma < 1 - 2 ( 1-s_{k})$. Put $ \sigma_{k} = 1 - 2(1-s_{k})-\frac{1}{10k}$. 
Then we have 
$$
\sup_{k} [f_{k}]_{\dot{F}^{\sigma_{k}}_{p,2}} \lesssim_{n,p,q} \Lambda.
$$
Then the result follows from Lemma \ref{WTL}.

 \end{proof}

\section{estimates for $W^{s}_{p,q}$}
\subsection{ Upper bounds:} 

In this subsection we prove the upper bounds appearing in Theorem \ref{1.9} and Theorem \ref{1.10}  regarding controlling $ W^{s}_{p,q}$ with a Triebel-Lizorkin semi-norm when $ s \in ( \theta , 1)$, for some $ \theta >0$. Namely we have
\begin{prop}\label{UWs}
Let $ p, q \in (1,\infty)$ and $ p > \frac{nq}{n+\theta q}$ for some $ \theta \in (0,1)$. Then for $ s \in (\theta,1)$ one has for $ f \in \SSS(\R^{n})$
\begin{enumerate}
    \item if $ q \in (1,2]$ then
    $$
    [f]_{W^{s}_{p,q}} \lesssim_{n,p,q,\theta} (1-s)^{-1/q} [f]_{\dot{F}^{s}_{p,q}};
    $$
    \item If $ q \in [2, \infty)$ then
    $$
    [f]_{W^{s}_{p,q}} \lesssim_{n,p,q,\theta} (1-s)^{-1/2}[f]_{\dot{F}_{p,q}};
    $$
    \item If $ q \in [2,\infty) $ then
    $$
    [f]_{W^{s}_{p,q}} \lesssim_{n,p,q,\theta} (1-s)^{-1/q} [f]_{\dot{F}^{s}_{p,2}}.
    $$
\end{enumerate}
\end{prop}

It is worth mentioning that one can prove the above proposition  by running the same type of arguments presented in Section 3. However, to avoid repetition we prove the more general result which is Theorem \ref{A2}. Then Proposition \ref{UWs} will be a consequence of the  upper bounds obtained for $ E^{s}_{p,q}$. 

Before we proceed, throughout this section we write $ f_{j} = \phi_{j} *f $ where $ \phi_{j}$ is as before, the function with Fourier transform given by 
$$
\mathcal{F}(\phi_{j}) = \vert 2^{-j} \xi \vert^{2} e^{-2 \pi \vert 2^{-j} \xi \vert}.
$$
The proof of Theorem \ref{A2} will be a sequence of propositions.

\begin{lemma}\label{KLW}
let $ p,q \in (1,\infty)$ and $ s \in (0,1)$ then we have 
$$
[f]^{p}_{W^{s}_{p,q}} \lesssim_{n,p,q} I + II + III
$$
where 
$$
I = \int_{\mathbb{R}^{n} } \left( \int_{\mathbb{R}^{n}} \vert z \vert^{-n-sq} \left( \int_{10 \vert z \vert}^{\infty} t^{3} \vert \partial^{2}_{t} P_{t} * f(x+z) - \partial^{2}_{2} P_{t} *f(x) \vert^{2} (x) dt \right)^{q/2} dz \right)^{p/q} dx,
$$
and 
$$
II = \int_{\mathbb{R}^{n} } \left( \int_{\mathbb{R}^{n}} \vert z \vert^{-n-sq} \left( \int_{0}^{10 \vert z \vert } t^{3} \vert \partial^{2}_{t} P_{t} * f(x)\vert^{2} (x) dt \right)^{q/2} dz \right)^{p/q} dx,
$$
and 
$$
III = \int_{\mathbb{R}^{n} } \left( \int_{\mathbb{R}^{n}} \vert z \vert^{-n-sq} \left( \int_{0}^{10\vert z \vert} t^{3} \vert \partial^{2}_{t} P_{t} * f(x+z)\vert^{2} (x) dt \right)^{q/2} dz \right)^{p/q} dx. 
$$
\end{lemma}
\begin{proof}
To simplify notation, write $ \delta_{z}(f)(x) = f(x+z) - f(x)$. Then we apply Corollary \ref{NK2} to the function $ \vert z\vert^{-s} \delta_{z} f(x) $ to get 
\begin{equation*}
\begin{split}
& [f]^{p}_{W^{s}_{p,q}} = \int_{\mathbb{R}^{n}}  \left( \int_{\R^{n}} \vert z \vert^{-n-sq} \vert \delta_{z}(f)(x) \vert^{q} dz \right)^{p/q} dx
\\
& \approx_{n,p,q} \int_{\mathbb{R}^{n}} \left( \int_{\R^{n}} \vert z \vert^{-n-sq} \big[ \int_{0}^{\infty} t^{4-1} \vert \partial^{2}_{t} \delta_{z} (P_{t} *f )(x) \vert^{2} dt  \big]^{q/2} dz \right)^{p/q} dx
\end{split}
\end{equation*}
Now we apply the triangle inequality to get the desired result.
\end{proof}
To get the conclusion of Theorem \ref{A2} we estimate $I, II$ and $III$ separately.  We start with the proof $I$. We will need the so called Cauchy estimates for harmonic functions which we recall now; given $u : \Omega \subset \mathbb{R}^{n} \to \mathbb{R}$ which is harmonic, one has the following inequality
\begin{equation}\label{Cauchyestimates}
    \vert \nabla u (x) \vert \lesssim r^{-1} \dashint_{B(x,r)} | u(y) | dy
\end{equation}
for any $ x \in \Omega$ and any $ r > 0 $ so that $ B(x,r) \subset \Omega$. For a proof of this estimate, look at \cite[Page 27]{evans10}.

\begin{prop}\label{P1}
Let $ p,q \in (1,\infty)$ and $ s \in (0,1)$. Then we have for $ f \in \SSS(\R^{n}) $ 
$$
I \lesssim_{n,p,q} [f]_{E^{s}_{p,q}} 
$$

\end{prop}
\begin{proof}
Denote for each fixed $ z $ the function $ G(x,t) = F(x,t) -F(x+z,t) $. Where $ F(x,t) = \partial_{t}^{2} P_{t} * f (x)$. Then note that $ F $ is harmonic in $ \mathbb{R}^{n+1}_{+}$. Further, note that 
$$
G(x,t) = \int_{0}^{1} \nabla F( x+ \sigma z,t) . z\mbox{ } d\sigma
$$
Denote by $ B^{n+1}$ a ball in $ \mathbb{R}^{n+1}$. Then by the Cauchy estimates for harmonic functions, \eqref{Cauchyestimates},  we have
\begin{equation}\label{Cauchy}
\begin{split}
& \vert G(x,t) \vert \leq \vert z \vert \int_{0}^{1} \vert \nabla F(x+\sigma  z  ,t) \vert d \sigma 
\\
& \overset{\eqref{Cauchyestimates}}{\lesssim}  \vert z \vert \int_{0}^{1} \frac{1}{t^{n+2}} \int_{B^{n+1}( (x+\sigma z,t), t/4)} \vert F(y,r) \vert dydr d \sigma
\end{split}
\end{equation}
Moreover, notice that for any $ \sigma \in (0,1)$ and $z \in \mathbb{R}^{n}$ with $ 10 | z| \leq t$  we have that $ B^{n+1}( (x+\sigma z,t) ,\frac{t}{4}) \subset B(x,t) \times (t/8, 4t)  $. And therefore we bound the right hand side of \eqref{Cauchy} as follows
\begin{equation}
\begin{split}
& \vert z \vert \int_{0}^{1} \frac{1}{t^{n+2}} \int_{B^{n+1}( (x+\sigma z,t), t/4)} \vert F(y,r) \vert dydr d \sigma
\\
&
\lesssim \vert z \vert \int_{0}^{1} \frac{1}{t^{n+2}} \int_{t/8}^{4t} \int_{B(x,t)} \vert F(y,r) \vert dy dr d\sigma
 \\
& 
\lesssim \vert z \vert \int_{0}^{1} \frac{1}{t^{n+2}}  \int_{B(x,t)} \int_{t/8}^{4t} \vert F(y,r) \vert dr dy d \sigma
\\
&
\overset{\eqref{maximalpoissonest}}{\lesssim} \frac{\vert z \vert}{t^{n+1}} \int_{B(x,t)} \mathcal{M}(F(.,t/9))(y) dy
 \\
& 
\frac{\vert z \vert}{ t} \mathcal{M} \big[ \mathcal{M} (F(., t/9)) \big] (x).
\end{split}
\end{equation}
To simplify notation we write $ F_{t}(x) = \mathcal{M}( F(., t/9))(x)$. 
Now we plug back to $I$ to get 
$$
I \lesssim \int_{\mathbb{R}^{n}} \left( \int_{\mathbb{R}^{n} } \vert z \vert^{-n+ (1-s)q} \left( \int_{10 \vert z \vert }^{\infty} t^{3}\vert \frac{\vert z\vert}{t}  \mathcal{M}(F_{t})\vert^{2} dt \right)^{q/2} dz \right)^{p/q} dx.
$$
Switching to polar coordinates and simplifying this becomes
$$
I \lesssim
\int_{\mathbb{R}^{n}} \left( \int_{0 }^{\infty}  r^{ (1-s)q-1} \left( \int_{10 r }^{\infty} t \vert \mathcal{M}(F_{t})\vert^{2} dt \right)^{q/2} dr \right)^{p/q} dx.
$$
Using Lemma \ref{CSF} and the remarks that follows it we get 
$$
I \lesssim_{n,p,q} \int_{\mathbb{R}^{n}} \left( \int_{0 }^{\infty}  r^{ (1-s)q-1} \left( \int_{10 r }^{\infty} t \vert F_{t} \vert^{2} dt \right)^{q/2} dr \right)^{p/q} dx.
$$
But by definition $ F_{t} (x)= \mathcal{M}( \partial^{2}_{t} P_{t/9} * f) (x)$. Hence, using a change of variable first, and then using Lemma \ref{CSF} again, we obtain
$$
I \lesssim_{n,p,q} \int_{\mathbb{R}^{n}} \left( \int_{0 }^{\infty}  r^{ (1-s)q-1} \left( \int_{ r }^{\infty} t \vert \partial^{2}_{t} P_{t} * f(x) \vert^{2} dt \right)^{q/2} dr \right)^{p/q} dx.
$$
Finally, the result follows by Proposition \ref{M}.

\end{proof}

Next, we estimate $II$.
\begin{prop}\label{P2}
Assume $ p,q \in (1,\infty)$ and $ s \in ( \theta ,1)$. Then we have
$$
[f]_{W^{s}_{p,q}} \lesssim_{n,p,q,\theta} [f]^{p}_{\dot{F}^{s}_{p,q}}
$$
\end{prop}
\begin{proof}
By definition of $II$ we have 
$$
II
= \int_{\mathbb{R}^{n} } \left( \int_{\mathbb{R}^{n}} \vert z \vert^{-n-sq} \left( \int_{0}^{10 \vert z \vert} t^{3} \vert \partial^{2}_{t} P_{t} * f(x)\vert^{2} dt \right)^{q/2} dz \right)^{p/q} dx 
$$
now we partition the inner integral for when $ 2^{-j-1} \leq t \leq 2^{-j}$ and the middle integral for when $ 2^{-k-1} \leq |z| \leq 2^{-k}$. Further, since $ |t| \leq 10 |z|$ then the sum over $j$ runs at most over $ j > k-10$. 
\begin{equation}\label{extradet}
\begin{split}
& \int_{\mathbb{R}^{n} } \left( \int_{\mathbb{R}^{n}} \vert z \vert^{-n-sq} \left( \int_{0}^{10 \vert z \vert} t^{3} \vert \partial^{2}_{t} P_{t} * f(x)\vert^{2}  dt \right)^{q/2} dz \right)^{p/q} dx 
\\
&\approx \int_{\mathbb{R}^{n}} \left( \sum_{k \in \mathbb{Z}} 2^{ksq} \left( \sum_{j > k-10} \int_{2^{-j-1}}^{2^{-j}} 2^{-3j} | \partial^{2}_{t} P_{t} * f(x) |^{2} dt \right)^{q/2} dz \right)^{p/q} dx.
\end{split}
\end{equation}
Now using \eqref{maximalpoissonest} and the definition of $f_{j}$, we bound the right hand side of \eqref{extradet} as follows
\begin{equation*}
\begin{split}
& \int_{\mathbb{R}^{n}} \left( \sum_{k \in \mathbb{Z}} 2^{ksq} \left( \sum_{j > k-10} \int_{2^{-j-1}}^{2^{-j}} 2^{-3j} | \partial^{2}_{t} P_{t} * f(x) |^{2} dt \right)^{q/2} dz \right)^{p/q} dx
\\
& \overset{\eqref{maximalpoissonest}}{\lesssim} \int_{\mathbb{R}^{n}} \left( \sum_{k \in \mathbb{Z}} 2^{ksq} \big[ \sum_{j > k-10} \vert \mathcal{M}(f_{j})(x) \vert^{2} \big]^{q/2} \right)^{p/q} dx.
\end{split}
\end{equation*}
Take $ \sigma = s/2$, then note that we have
\begin{equation}\label{Peq}
\begin{split}
& \big[ \sum_{j > k-10 } \vert \mathcal{M}(f_{j})(x) \vert^{2} \big]^{q/2}
\\
& \leq \sup_{j > k-10} \vert 2^{\sigma j} \mathcal{M}(f_{j}) \vert^{q} \big[\sum_{j>k-10} 2^{-2\sigma j }\big]^{q/2}
\\
& \lesssim_{n,q,\theta} 2^{-k\sigma q} \sum_{j >k-10} 2^{\sigma j q} \vert \mathcal{M}(f_{j})(x) \vert^{q}.
\end{split}
\end{equation}
Using this, we obtain 
\begin{equation*}
\begin{split}
& II 
\lesssim_{n,p,q,\theta} \int_{\mathbb{R}^{n}} \left( \sum_{k \in \mathbb{Z}} 2^{k(s-\sigma)q} \sum_{j> k-10} 2^{j\sigma q} \vert \mathcal{M}(f_{j})(x) \vert^{q} \right)^{p/q} dx
\\
& = \int_{\mathbb{R}^{n}} \left( \sum_{j\in \mathbb{Z}} 2^{j\sigma q} \vert \mathcal{M}( f_{j})(x) \vert^{q} \sum_{k<j+10} 2^{k(s-\sigma)q} \right)^{p/q} dx
\\
& \lesssim_{n,p,\theta,q} \int_{\mathbb{R}^{n}} \left( \sum_{j\in \mathbb{Z} } 2^{jsq} \vert \mathcal{M}(f_{j})(x) \vert^{q} \right)^{p/q} dx.
\end{split}
\end{equation*}
Now the result follows from Lemma \ref{SteinF} and Lemma \ref{PT}. ( Here we apply Lemma \ref{SteinF} to $ \tilde{f}_{j} : = 2^{js} f_{j}$)
\end{proof}
Finally, we estimate $III$. Notice that for $ I,II$ we did not need the assumption that $ p > \frac{nq}{n+\theta q}$. However, the situation for $III$ is different. The difficulty comes from the fact that we have a middle integral in the variable $z$. And to get rid of that integral we need to apply the Peetre maximal function with the correct exponent.

\begin{prop}\label{P3}
Let $ p,q \in (1,\infty)$ and $ s \in (\theta,1)$ for  some $ \theta >0$. Assume further that $ p > \frac{nq}{n+\theta q}$. Then we have 
$$
III \lesssim_{n,p,q,\theta} [f]^{p}_{\dot{F}^{s}_{p,q}}
$$
\end{prop}
\begin{proof}
Since $ p > \frac{nq}{n+\theta q}$, then we can find $ \beta = \beta(p,q,n,\theta) \in (0,1)$ so that one has 
$$
p > \frac{nq}{n+\theta \beta q}.
$$
Then take $ \sigma = (1-\beta)s$. 
By definition of $III$ we have 
$$
III=\int_{\mathbb{R}^{n} } \left( \int_{\mathbb{R}^{n}} \vert z \vert^{-n-sq} \left( \int_{0}^{10\vert z \vert} t^{3} \vert \partial^{2}_{t} P_{t} * f(x+z)\vert^{2} (x) dt \right)^{q/2} dz \right)^{p/q} dx.
$$
Similar to what has been done in \eqref{extradet}, we partition the integral in $z$ for when $ 2^{-k-1} \leq |z| \leq 2^{-k}$ and the integral in $ t$ for when $ 2^{-j-1} \leq t \leq 2^{-j}$ then since $ t \leq 10 |z|$ we have that the sum in $j$ runs at most over $ j > k-10$. Then
using Lemma \ref{Psn}, \eqref{maximalpoissonest}, we get 
$$
III \lesssim \int_{\mathbb{R}^{n}} \left( \sum_{k \in \mathbb{Z}}2^{ksq} \dashint_{\vert z \vert \approx 2^{-k}} \big[ \sum_{j>k-10} \vert \mathcal{M} ( f_{j})(z+x) \vert^{2} \big]^{q/2} dz \right)^{p/q} dx. 
$$
Use the notation $ m_{j}(x+z) = \mathcal{M} ( f_{j})(z+x)$. Then
arguing as in \eqref{Peq} we have 
\begin{equation}
\begin{split}
& \big[ \sum_{j>k-10} \vert m_{j}(z+x) \vert^{2} \big]^{q/2}
\\
& \lesssim_{n,p,q,\theta} 2^{-k\sigma q} \sum_{j > k-10} 2^{j\sigma q} \vert m_{j}(x+z) \vert^{q}. 
\end{split}
\end{equation}
Since we got rid of the power $ q/2$, we next estimate 
$$
\dashint_{\vert z \vert \approx 2^{-k}} \vert m_{j}(x+z) \vert^{q} dz.
$$
We do so as follows, take $ r \in (\frac{nq}{n+\beta \theta q} , \min\{p,q\})$. Then note the obvious pointwise bound for when $  z \approx 2^{-k}$ and $ j > k$. 
\begin{equation}
\begin{split}
& m_{j}(x+z) \lesssim (1+ 2^{(j-k)})^{\frac{n}{r}} \sup_{z \in \mathbb{R}^{n}} \frac{ m_{j}(x+z)}{ ( 1+ 2^{j} \vert z \vert)^{n/r}}
\\
& \lesssim 2^{(j-k)\frac{n}{r}} m^{*}_{j}(x).
\end{split}
\end{equation}
Where $ m_{j}^{*}$ is the Peetre maximal function of the function $m_{j}$. Then using this pointwise bound we obtain
\begin{equation*}
\begin{split}
&\dashint_{\vert z \vert \approx 2^{-k}} \vert m_{j}(x+z) \vert^{q} dz 
\\
&\lesssim 2^{(j-k)\frac{n}{r}(q-r)} \vert m^{*}_{j}(x)\vert^{(q-r)} \dashint_{\vert z \vert \approx 2^{-k}} \vert m_{j}(x+z) \vert^{r} dz
\\
& \leq 2^{(j-k)\frac{n}{r}(q-r)}  \vert m^{*}_{j}(x)\vert^{(q-r)} \mathcal{M}(\vert m_{j} \vert^{r} )(x) 
\end{split}
\end{equation*}
Young's product inequality, with the exponents $ \frac{q}{r}$ and $ \frac{q}{q-r}$, gives us 
\begin{equation}\label{youngsprod}
\dashint_{\vert z \vert \approx 2^{-k}} \vert m_{j}(x+z) \vert^{q} dz \lesssim_{q} 2^{(j-k)\frac{n}{r}(q-r)}\vert m_{j}^{*}(x) \vert^{q} + 2^{(j-k)\frac{n}{r}(q-r)} \mathcal{M}(\vert m_{j} \vert^{r})(x)^{q/r}.
\end{equation}
Using \eqref{youngsprod} we obtain
\begin{equation*}
\begin{split}
&III
\lesssim_{n,p,q,\theta} \int_{\mathbb{R}^{n}} \left( \sum_{k\in \mathbb{Z}}2^{k(s-\sigma)q} \sum_{j>k-10} 2^{\sigma j q} 2^{(j-k)\frac{n}{r}(q-r)} \vert m^{*}_{j}(x)\vert^{q} \right)^{p/q} dx
\\
&
+ \int_{\mathbb{R}^{n}} \left( \sum_{k\in \mathbb{Z}} 2^{k(s-\sigma)q} \sum_{j>k-10} 2^{\sigma j q} 2^{(j-k)\frac{n}{r}(q-r)}  \mathcal{M}(\vert m_{j} \vert^{r})(x)^{q/r} \right)^{p/q} dx
\\
&
= \Gamma_{1} + \Gamma_{2}.
\end{split}
\end{equation*}
Now we estimate each term separately. 
\begin{equation*}
\begin{split}
& \Gamma_{1} =\int_{\mathbb{R}^{n}} \left( \sum_{k\in \mathbb{Z}} 2^{k(s-\sigma)q} \sum_{j>k-10} 2^{\sigma j q} 2^{(j-k)\frac{n}{r}(q-r)} \vert m^{*}_{j}(x)\vert^{q} \right)^{p/q} dx
\\
& = \int_{\mathbb{R}^{n}} \left( \sum_{j \in \mathbb{Z}} \vert m^{*}_{j}(x) \vert^{q} 2^{j\sigma q} 2^{j \frac{n}{r}(q-r)}  \sum_{k < j+10} 2^{k(s - \sigma)q} 2^{-k \frac{n}{r}(q-r)} \right)^{p/q} dx.
\end{split}
\end{equation*}
By the choice of $ \sigma $ we have that $ s- \sigma = \beta s$. Further, by our choice of $ r \in ( \frac{nq}{n+\theta \beta q}, \min\{p,q\})$ we see that 
$$
\beta sq - \frac{n}{r}(q-r) >\beta \theta q - \frac{n}{r}(q-r) > 0.
$$
Hence, evaluating the inner sum we conclude that 
$$
\Gamma_{1} \lesssim_{q,p,n,\theta} \int_{\mathbb{R}^{n}} \left( \sum_{j \in \mathbb{Z}} 2^{jsq} \vert m^{*}_{j}(x) \vert^{q} dx \right)^{p/q} dx.
$$
Now the result follows by Lemma \ref{PetreeM} and then Lemma \ref{SteinF}. We proceed to estimate $ \Gamma_{2}.$  Using the same reasoning
\begin{equation*}
\begin{split}
& \Gamma_{2} =\int_{\mathbb{R}^{n}} \left( \sum_{k\in \mathbb{Z}} 2^{k(s-\sigma)q} \sum_{j>k-10} 2^{\sigma j q} 2^{(j-k)\frac{n}{r}(q-r)}  \mathcal{M}(\vert m_{j}\vert^{r})(x)^{q/r} \right)^{p/q} dx
\\
& = \int_{\mathbb{R}^{n}} \left( \sum_{j \in \mathbb{Z}} \mathcal{M}(\vert m_{j} \vert^{r})(x)^{q/r} 2^{j\sigma q} 2^{j \frac{n}{r}(q-r)}  \sum_{k < j+10} 2^{k(s - \sigma)q} 2^{-k \frac{n}{r}(q-r)} \right)^{p/q} dx
\\
&\lesssim_{p,q,\theta,n} \int_{\mathbb{R}^{n}} \left( \sum_{j \in \mathbb{Z}} \mathcal{M}(\vert m_{j} \vert^{r})(x)^{q/r} 2^{jsq}\right)^{p/q} dx,
\end{split}
\end{equation*}
and the result follows by applying Lemma \ref{SteinF} twice.

\end{proof}
The proof of Theorem \ref{A2} is clear now.
\begin{proof}[Proof of Theorem \ref{A2}]
We have from Proposition \ref{P1}, Proposition \ref{P2}, and Proposition \ref{P3} that for $ f \in \SSS(\R^{n})$ and $ p > \frac{nq}{n+ \theta q}$ the following inequality holds
$$
[f]_{W^{s}_{p,q}} \lesssim_{n,p,q,\theta} [f]_{E^{s}_{p,q}} + [f]_{\dot{F}^{s}_{p,q}},
$$
the result then follows from Proposition \ref{L}. 
\end{proof}
As a corollary of Theorem \ref{A2} we get Theorem \ref{T1.11} and Proposition \ref{UWs}.

\subsection{Lower bounds} The main result in this subsection are the lower bounds in Theorem \ref{1.9} and Theorem \ref{1.10}.
\begin{prop}\label{UUU}
Let $ p,q \in (1, \infty)$ and $ s \in ( \theta, 1)$ for some $ \theta > 0$. Assume further, that $ p' > \frac{nq'}{n+\theta q'}$. Then we have 
\begin{enumerate}
    \item if $ q \in (1,2]$ then we have 
    $$
    [f]_{\dot{F}^{s}_{p,q}} \lesssim_{n,p,q,\theta} (1-s)^{1/2} [f]_{W^{s}_{p,q}}.
    $$
    \item If $ q \in [2,\infty)$ then we have 
    $$
    [f]_{\dot{F}^{s}_{p,q}} \lesssim_{n,p,q,\theta} (1-s)^{1/q} [f]_{W^{s}_{p,q}}.
    $$
    \item if $ q \in (1,2]$ then we have 
    $$
    [f]_{\dot{F}^{s}_{p,2}} \lesssim_{n,p,q,\theta} (1-s)^{1/q} [f]_{W^{s}_{p,q}}.
    $$
\end{enumerate}
\end{prop}
\begin{proof}
Using Lemma \ref{DualityForW} we get 
\begin{equation}
\begin{split}
    [f]_{\dot{F}^{s}_{p,q}} \lesssim_{n,p,q} (1-s)[f]_{W^{s}_{p,q}} \sup_{g} [g]_{W^{s}_{p',q'}}
\end{split}
\end{equation}
Where the sup is taken over all $ g \in C^{\infty}_{0}$ with norm $[g]_{\dot{F}^{s}_{p',q'}} \leq 1$. Fix any such $g$, then using Proposition \ref{UWs}, which we can use since $ p' > \frac{nq'}{n+\theta q'}$,  we get 
$$
[g]_{W^{s}_{p',q'}} \lesssim_{n,p,q,\theta} (1-s)^{-1/\sigma}.
$$
Where $ \sigma =\min\{ 2,q'\}$. Then we plug this back into (4.6) above to get (1),(2). Also for (3) we have 
$$
[f]_{\dot{F}^{s}_{p,2}} \lesssim_{n,p,q} (1-s)[f]_{W^{s}_{p,q}} \sup_{g} [g]_{W^{s}_{p',q'}}.
$$
Where the supremum is taken with the indicated constraints. Then we use the upper bounds in Proposition \ref{UWs} to conclude. 
\end{proof}
Next, we prove an analogous result of Theorem \ref{LowerS}. 
\begin{prop}\label{LWS}
Let $ p,q \in (1,\infty)$, and let $ \Theta, \gamma >1  $, put $ \theta =1 -\frac{1}{\gamma}$ and $ s \in (1-\frac{1}{ 2\gamma \Theta}, 1)$. Assume further that $ p' > \frac{nq'}{n+\theta q'}$. Then for $ \sigma \in (0, \bar{\sigma})$, where $ \bar{\sigma}$ is taken so that $ (1- \bar{\sigma}) = \Theta (1-s)$, we have 
$$
[f]_{F^{\sigma}_{p,2}} \leq C(\Theta,\theta, n, p,q ) \left( \Vert f \Vert_{L^{p}} + {(1-s)^{1/q}} [f]_{W^{s}_{p,q}} \right).
$$
\end{prop}
\begin{proof}
Since for $ \sigma \in (0, \bar{\sigma})$ we have 
$$
[f]_{\dot{F}^{\sigma}_{p,2}} \lesssim_{n,p} \Vert f \Vert_{L^{p}} + [f]_{\dot{F}^{\bar{\sigma}}_{p,2}}.
$$
It suffices to prove the theorem for $ \sigma = \bar{\sigma}$. By Lemma \ref{DualityForW} we have 
$$
[f]_{\dot{F}^{\bar{\sigma}}_{p,2}} \lesssim_{p,q,n} [f]_{\dot{F}^{0}_{p,2}}+  (1-\bar{\sigma}) [f]_{W^{s}_{p,q}}  \sup_{\substack{g \in C^{\infty}_{0}\\ \text{supp}(\mathcal{F}(g) ) \subset \{ \vert \xi \vert >1/4 \} \\ [g]_{\dot{F}^{\bar{\sigma}}_{p',2} \leq 1}}}[g]_{W^{2\bar{\sigma}-s}_{p'.q'}}.
$$
Fix a $g$ with the indicated constraints, then note that
$$
 \bar{\sigma} > 2 \bar{\sigma}-s \geq 1- \frac{1}{\gamma}= \theta,
$$
and since by the hypothesis we have $ p' > \frac{nq'}{n+\theta q'}$ we get, by using Theorem \ref{T1.11}
$$
[g]_{W^{2\bar{\sigma}-s}_{p'.q'}}. \lesssim_{n,p,q,\theta} (2\bar{\sigma}-s)^{-1/q'} [g ]_{\dot{F}^{0}_{p',2}} + (s- \bar{\sigma})^{-1/q'} [g]_{\dot{F}^{\bar{\sigma}}_{p',2}}.
$$
Now the support condition on $ g $ ensures that we have 
$$
[g]_{\dot{F}^{0}_{p',2}} \lesssim_{n,p} [g]_{\dot{F}^{\bar{\sigma}}_{p',2}}.
$$
Using this we obtain,
\begin{equation*}
\begin{split}
& [f]_{\dot{F}^{\bar{\sigma}}_{p,2}}  \lesssim_{n,p,q,\theta} \Vert f \Vert_{L^{p}} +(1-\bar{\sigma}) \left( (2\bar{\sigma}-s)^{-1/q'} + (s-\bar{\sigma})^{-1/q'} \right) [f]_{W^{s}_{p,q}}.
\end{split}
\end{equation*}
And this gives the result since
$$
(1-\bar{\sigma}) \left( (2\bar{\sigma}-s)^{-1/q'} + (s-\bar{\sigma})^{-1/q'} \right) \lesssim_{\Theta,n,q} (1-s)^{1/q}.
$$
Indeed, since 
$$
2 \bar{\sigma}-s = 2 - 2 \Theta (1-s) -s \geq ( 1- \frac{1}{\gamma})+ 1-s \geq 1-s.
$$
\end{proof}

From the above we can conclude Corollary \ref{BWs}. 
\begin{proof}[Proof of Corollary 1.13]
Since $ p ' > \frac{nq'}{n+q'}$ then we can find $ \theta \in (0,1) $ that depends only on $ p,q, n $ so that
$$
p ' > \frac{nq'}{n+\theta q'}.
$$
Then define $ \gamma = (1- \theta)^{-1}$. Then $ \gamma > 1 $, and $ \theta = 1- \frac{1}{ \gamma} $. Now upon removing finitely many $ s_{k}$'s we may assume that $ s_{k} \in ( 1- \frac{1}{4\gamma},1) $.
Therefore, we apply Proposition \ref{LWS} above with $ \Theta = 2 $  to get 
$$
[f]_{\dot{F}^{\sigma}_{p,2}} \lesssim_{n,p,q} \Vert f \Vert_{L^{p}} + ( 1- s_{k})^{1/q}[f]_{W^{s_{k}}_{p,q}} \leq \Lambda,
$$
for any $ \sigma < 1- 2 ( 1-s_{k})$. Hence the result follows by Lemma \ref{WTL}.

\end{proof}

\section{Relation between $W^{s}_{p,q}$ and $E^{s}_{p,q}$}
Here we prove Theorem \ref{A1} 
\begin{proof}[Proof of Theorem \ref{A1}]
From Theorem \ref{A2} we have 
$$
[f]_{W^{s}_{p,2}} \lesssim_{n,p,\theta} [f]_{E^{s}_{p,2}}.
$$
For the other direction we have 
\begin{equation*}
\begin{split}
& [f]_{E^{s}_{p,2}} \lesssim_{n,p,q} (1-s)^{-1/2} [f]_{\dot{F}^{s}_{p,2}} 
\\
& \lesssim_{n,p,\theta} (1-s)^{-1/2} (1-s) [f]_{W^{s}_{p,2}} \sup_{g} [g]_{W^{s}_{p',2}}.
\end{split}
\end{equation*}
Where the sup is taken over $ g \in C^{\infty}_{0}$ with norm $1$ in $\dot{F}^{s}_{p',2}$. Fix any such $g$ and using the upper bounds for $W^{s}_{p,q}$ we obtain
$$
[f]_{E^{s}_{p,2}} \lesssim_{n,p,q,\theta} (1-s)^{1/2}  [f]_{W^{s}_{p,2}} (1-s)^{-1/2}[g]_{\dot{F}^{s}_{p',2}}.
$$
But since $ g$ has norm $1$ as an element of $\dot{F}^{s}_{p',2}$ this gives the desired inequality.

\end{proof}

\section{Appendix A: Triebel-Lizorkin space via the Poisson kernel}
 In this section we sketch the proof of Lemma \ref{PT}. Of course nothing here is new, we just keep track of the dependency on $ s $ in the proof of Lemma \ref{PT} provided in \cite[Theorem 1.4]{CB}. We will not repeat the arguments given in \cite{CB}, we will merely indicate the dependency on $ s $ in their proof. In their proof, they work with an abstract kernel $ \psi$. In our case we are only concerned with the Poisson kernel; therefore, any lemma we reference we adapt it to the Poisson kernel.
Throughout this subsection $ f \in C^{\infty}_{0}$. 
\begin{lemma}\label{6.1}
Take any $ p,q \in (1,\infty)$ and let $ \phi(x) = \partial^{2}_{t} P_{1}(x).$ And denote by $ \varphi_{j}$ to mean the same function in Definition \ref{DF}. Take any $ m > \max \{ \frac{n}{p}, \frac{n}{q}\}$.
Then we have the pointwise inequality 
$$
\vert \varphi_{j}*\phi(x) \vert \lesssim_{n,p,q}\begin{cases}
2^{-jm}\frac{1}{(1+ \vert x \vert)^{[\Lambda]+n+1}}, & \text{ if } j \geq 0\\
2^{2j} 
\frac{2^{j}}{(1+ \vert x \vert)^{[\Lambda]+n+1}}
, & \text{if } j < 0.
\end{cases}
$$
Where $ \Lambda = \max \{ n/p, n/q \} $. 
\end{lemma}
The proof of the above Lemma can be found in \cite[Lemma 2.4]{CB}.
Next, Bui and Candy prove a pointwise inequality that relates the Petree maximal function of $ \phi_{j}*f$ to that of $ \varphi_{j}*f$, see \cite[Page 36]{CB}. 
\begin{lemma}
Let $ p,q \in (1,\infty)$ and $ \Lambda = \max \{ n/p, n/q \}$ as above, and $ \lambda > \Lambda$. Then if we denote by $ \phi^{*}_{j}f $ to mean the Petree maximal function in Definition \ref{PMax}, and similarly $ \varphi^{*}_{j}f$. Then we get 
$$
2^{js} \phi^{*}_{j}f(x) \lesssim_{n,p,q} \sum_{j \in \mathbb{Z}} a_{j-k} \varphi^{*}_{j}f(x).
$$
Where $ s \in \mathbb{R}$, and 
$$
a_{j} = 2^{-js} \sup_{x} \int_{\mathbb{R}^{n}} \vert \phi * \varphi_{j} (x) \vert \frac{(1+ 2^{j} \vert x+ y \vert)^{\lambda} }{(1+ \vert x \vert)^{\lambda}} dy.
$$
\end{lemma}
With the two results above we can proceed to prove Lemma \ref{PT}. 
Using first Lemma \ref{6.1} one get, see \cite[Page 37]{CB}
$$
a_{j} \lesssim_{n,p,q}  \begin{cases}
2^{j(\lambda -m - s)} \int_{\mathbb{R}^{n}} (1+ \vert y \vert)^{\lambda- n -1 - [\Lambda]}, & \text{ if } j \geq 0\\
2^{j(2-s)} \int_{\mathbb{R}^{n}} 2^{jn}( 1 + 2^{j} \vert y \vert)^{-n-1 - [\Lambda]} dy 
, & \text{if } j < 0.
\end{cases}
$$
Using this we can prove Lemma \ref{PetreeM}, which will imply one direction of Lemma \ref{PT}. First, since the Petree maximal function is a decreasing function in $ \lambda$, we may assume that $ \Lambda < \lambda < \Lambda +1 $. Then
 choosing $m > \lambda +1 $ and $ s \in [0,1]$ we get the following inequality 
$$
\sum_{j} \vert 2^{js} \phi^{*}_{j}f(x) \vert^{q} \lesssim_{n,p,q,\lambda} \left( \frac{1}{2-s} + \frac{1}{m-\lambda -s } \right) \sum_{j} \vert 2^{js} \varphi^{*}_{j}f(x) \vert^{q}.
$$
Since $ s \in [0,1]$ and $ m > \lambda +1 $ we get the inequality with no dependency on $ s $. Now the result follows once we raise both sides to power $ p/q $ and integrate in $ x $, then we use Petree characterization, \cite{Peetre}. This proves Lemma \ref{PetreeM}, and hence one direction of Lemma \ref{PT}. Finally, we prove the other direction. 
\begin{lemma}\label{6.3}
Let $ \lambda \in ( \Lambda , [\Lambda]+2 ) $, and $m > \lambda + 1 $. Then we have for $ r \in ( 0, \infty)$ the following inequality 
$$
\varphi^{*}_{j}f(x) \lesssim_{n,p,q,r} \sum_{k \geq j-c} 2^{(j-k)(m- \lambda)} \int_{\mathbb{R}^{n}} 2^{jn} \frac{ \vert \phi_{k} * f(x-y)\vert^{\lambda r}}{(1+ 2^{k} \vert y \vert)^{\lambda r}}.
$$
Where $ c $ depends only on $ n$.
\end{lemma}
The lemma above is \cite[Corollary 4.3]{CB}. The authors then use this pointwise inequality with $ r \in ( 0, \min \{ p,q \})$ so that $ \frac{n}{r} \in ( \Lambda, \lambda)$. In particular, $ r \lambda >n$. Then if we use the known inequality 
$$
 \int_{\mathbb{R}^{n}}2^{jn} \frac{ \vert g(x-y)}{(1+ 2^{k} \vert y \vert)^{N}} \lesssim_{N,n} \mathcal{M}(g)(x).
 $$
Which holds whenever $ N >n$, then we combine this with Lemma \ref{6.3} to get 
$$
( 2^{js} \varphi^{*}_{j}f(x) )^{r} \lesssim_{n,p,q} \sum_{k > 1-c} 2^{-k(m+s-\lambda)} \mathcal{M}(( 2^{(k+j)s} \vert \phi_{k+j}*f \vert^{r}))(x)
$$
Raising both sides to power $ q/r$ and summing over $j$ we get
$$
\sum_{j} \vert 2^{js} \varphi^{*}_{j}f(x)\vert^{q} \lesssim_{n,p,q} \sum_{j} \left( \sum_{k > 1-c} 2^{-k(m+s-\lambda)} \mathcal{M}(( 2^{(k+j)s} \vert \phi_{k+j}*f \vert^{r}))(x)
\right)^{q/r}.
$$
We raise both sides to power $ \frac{p}{q} $ and integrate in $ x $ to get, upon using Minkowki's inequality and Stein Fefferman inequality
$$
\int_{\mathbb{R}^{n}} \left(\sum_{j} \vert 2^{js} \varphi^{*}_{j}f(x)\vert^{q}\right)^{p/q} \lesssim_{n,p,q} \frac{1}{(m+s-\lambda)} \int_{\mathbb{R}^{n}} \left( \sum_{j} \vert \phi_{j}*f \vert^{q} \right)^{p/q}.
$$
Since $ m > \lambda +1$, then the inequality holds with no dependency on $ s$. 
This gives us the other direction.

\section{Appendix B: Bochner integrals and Littlewood Paley for mixed $L^{p}$ spaces}
In this appendix we provide the proof for Lemma \ref{Kformixed}. The proof is very similar to the Littlewood-Paley theorem for $L^{p}$ spaces. However, we need to use Bochner integrals and think of a function $ f \in L^{p,q}_{\alpha}( \R^{n} \times \Omega)$ as an element in the Bochner Lebesgue space $L^{p}( B)$ for an appropiate Banach space. 
\subsection{Bochner integrals}
Throughout this subsection let $B$ denote a separable reflexive Banach space, and denote by $B'$ its dual. Then for $ f: \mathbb{R}^{n} \to B$ we say that $ f $ is (strongly) measurable if for every linear functional $ b' \in B' $ we have that the scalar valued function $ b'(f(x)) $ is measurable as a function of $ x $. We define $ L^{p}(B)$ as the space of all (strongly) measurable functions $ f $ with 
$$
\int_{\mathbb{R}^{n}} | f(x) |^{p}_{B} dx < \infty.
$$
a more accurate notation would be $L^{p}(\mathbb{R}^{n}; B)$. However, since the domain will always be $\mathbb{R}^{n}$, we will always write $ L^{p}(B)$ instead of $L^{p}(\mathbb{R}^{n};B)$.
The following lemma can be found in \cite[Theorem 1.3.21]{H16}
\begin{lemma}\label{Bd}
Let $ p \in [1,\infty)$ then $L^{p}(B)$ is a Banach space with a dual given by $ L^{p'}(B')$ via the pairing 
$$
\langle g, f \rangle = \int_{\mathbb{R}^{n}} g(x) \cdot f(x) dx. 
$$
where $ g \in L^{p'}(B')$ and $ f \in L^{p}(B)$.
\end{lemma}
We will also need the following density result, which can be found in \cite[Proposition 5.5.6]{GC}
\begin{lemma}\label{density}
Let $ p \in [1,\infty)$. Denote by $L^{p}(\mathbb{R}^{n}) \otimes B$ to be the subspace of $ L^{p}(B)$ consisting of linear combinations of functions of the form $ F(x) = f(x) b $. Where $ b \in B$ and $ f $ is a real valued $L^{p}$ function. Then $L^{p}(\mathbb{R}^{n}) \otimes B$ is dense in $L^{p}(B)$.
\end{lemma}

For a comprehensive treatment of Bochner spaces we refer the reader to \cite{BI}, and \cite{H16}. In our case, we will only consider very specific Banach spaces. Namely, let $ m \in \mathbb{N}$, $ q \in (1,\infty)$ and let $ \Omega \subset \mathbb{R}^{m}$ be an open subset. Then consider a continuous, positive function 
$ \alpha: \Omega \to \mathbb{R} $. Then we define $B^{q}_{\Omega, m,\alpha}$ as follows 
$$
B^{q}_{\Omega,m,\alpha} = \{ f \in L_{loc}^{1}(\Omega) : \int_{\Omega} \vert f(r) \vert^{q} \alpha(r) dr < \infty \}. 
$$
Clearly $B^{q}_{\Omega,m,\alpha}$ is a Banach space under the norm 
$$
\vert f \vert_{B_{\Omega,m,\alpha}} = \left( \int_{\Omega} \alpha (r) \vert f(r) \vert^{q} dr\right)^{1/q}
$$ with dual $B^{q'}_{\Omega,m,\alpha} $. Indeed, $B_{\Omega,m,q}$ is nothing but the usual $L^{q}$ space with respect to the measure induced by $ \alpha$. Now for $p \in (1,\infty)$ we study the space $ L^{p}(B^{q}_{\Omega,m,\alpha})$. The next lemma is an immediate consequence of Lemma \ref{Bd}
\begin{lemma}
The dual of $L^{p}(B^{q}_{\Omega,m,\alpha})$ is the space $L^{p'}(B^{q'}_{\Omega,m,\alpha})$.
\end{lemma}
Notice that elements of $ L^{p}(B^{q}_{\Omega,m,\alpha})$ can be identified with 
 real valued functions with domain on $\mathbb{R}^{n} \times \Omega$ such that 
 $$
 \Vert \vert f \vert_{B^{q}_{\Omega , m , \alpha)}} \Vert_{L^{p}(\mathbb{R}^{n})} < \infty.
 $$
Such an identification can be realized by writing for each $ f \in L^{p}(B_{\Omega,m,\alpha}^{q})$ the corresponding function 
$$
F(x,r) = f(x)(r).
$$
Of course this is not entirely correct, since $ f(x) = [f(x,.)]$. That is for each $x$, $ f$ evaluates to an equivalence class of functions $ f(x,\cdot)$. But if we choose any representative then we can think of $ f $ as a function on $ \mathbb{R}^{n} \times \Omega$. However,
the issue here is that $F$ might not be measurable as a function defined on $ \mathbb{R}^{n} \times \Omega$, if we make pathological choices for each $x$. But this can be easily overcame, since we will only work with $ f \in C^{\infty}_{0}(\mathbb{R}^{n} \times \Omega) $ and we will show below that such functions are dense in $L^{p}(B^{q}_{\Omega,m,\alpha})$.
\begin{lemma}
Let $ p, q \in (1,\infty)$, then
$ C^{\infty}_{0}( \mathbb{R}^{n} \times \Omega) $  is dense in $ L^{p}(B^{q}_{\Omega, m , \alpha}) $.
\end{lemma}
\begin{proof}
First assume that $ f \in L^{p}(B^{q}_{\Omega,m,\alpha})$ is given by $ f(x)(r) = g(x) h(r)$. then we approximate in the natural way, that is find $ g_{i} \in C^{\infty}_{0}(\mathbb{R}^{n})$ and $ h_{i} \in C^{\infty}_{0}(\Omega)$ so that $ g_{i} \to g $ in $L^{p}(\mathbb{R}^{n})$ and $ h_{i} \to h $ in $L^{q}(\Omega , \mu)$ where $ \mu$ is the measure induced by $ \alpha.$ then clearly $ g_{i}h_{i} \to g.h =f $ in $L^{p}(B_{\Omega, m , \alpha}^{q})$. Now for general $ f $ the result follows in view of Lemma \ref{density}.
\end{proof}
The next result is in \cite[Page 106]{J}, see also \cite[Theorem 1.3]{BanachS}. First for Banach spaces $B_{1}, B_{2}$ denote by $\mathcal{L}(B_{1},B_{2})$ to mean the space of all bounded linear operators between $B_{1},$ and $B_{2}$.
\begin{lemma}\label{Singular}
Let $B_{1}, B_{2}$ be any separable, reflexive Banach spaces, and $$K: \mathbb{R}^{n} \times \mathbb{R}^{n} \setminus \Delta \to \mathcal{L}(B_{1},B_{2}).
$$
a measurable function 
with the property that 
\begin{enumerate}
    \item $$
    \int_{\vert x-y \vert > 2 \vert y -z \vert } \vert K(x,y) - K(x,z) \vert_{\mathcal{L}(B_{1},B_{2})} dx \leq C $$
    \item $$
    \int_{\vert x-y \vert > 2 \vert x -z \vert } \vert K(x,y) - K(z,y) \vert_{\mathcal{L}(B_{1},B_{2})} dx \leq C.$$   
\end{enumerate}
Define the operator $ T$ on the space $L^{p}(B_{1})$ as the operator with $K$ as its kernel, namely 
$$
T(f)(x) = \int_{\mathbb{R}^{n}} K(x,y).f(y) dy
$$
for $ x \notin \text{Supp}(f)$. Then if $ T $ is bounded from $L^{r}(B_{1})$ to $L^{r}(B_{2})$ for some $ r \in (1,\infty)$, then $T$ is a bounded operator from $L^{p}(B_{1})$ to $L^{p}(B_{2})$ for any $ p \in (1,\infty)$.
\end{lemma}

\subsection{Proof of Littlewood-Paley} 
We need to adapt the standard proof of Littlewood-Paley to appropriate Banach valued functions. We begin by defining those Banach spaces.
\begin{defn}
Let $ \Omega \subset \mathbb{R}^{m}$ be open, and $ \alpha $ a continuous positive function on $ \Omega $. Then for $ q \in (1,\infty)$ and $ l \in \mathbb{N}$ we define the Banach space $ A_{\Omega,\alpha,m}^{q,l} $ as follows
$$
A_{\Omega,\alpha,m}^{q,l} = \{ F \in L^{1}_{loc}(\Omega \times (0,\infty)) : \vert F \vert^{q}_{A^{q,l}_{\Omega,\alpha,m}} = \int_{\Omega}  \alpha(r) \left( \int_{0}^{\infty} t^{2l-1} \vert F(r,t) \vert^{2} dt \right)^{q/2} dr < \infty \}.
$$
\end{defn}
Next, we define operators between the spaces $B^{q}_{\Omega,\alpha,m}$ and $ A^{q,l}_{\Omega,\alpha,m}$. For the rest of the section fix $ \epsilon > 0$. 
\begin{lemma}\label{PKK}
Let $ l \in \mathbb{N}$. Define $ K^{\epsilon}_{l}(x) = \partial^{l}_{t} P_{t+\epsilon}(x)$. Then we claim that for each fixed $x$, $K^{\epsilon}_{l}(x)$ acts as a linear operator between $B^{q}_{\Omega,\alpha,m}$ and $ A^{q,l}_{\Omega,\alpha,m}$ via the operation
$$
K^{\epsilon}_{l}(x)(f)(t,r) = f(r) \partial^{l}_{t} P_{t+\epsilon}(x)
$$, that is we have $ K^{\epsilon}_{l}(x) \in \mathcal{L}(B^{q}_{\Omega,\alpha,m} , A^{q,l}_{\Omega,\alpha,m})$ with norm, when $ x \neq 0$,
$$
\vert K^{\epsilon}_{l}(x) \vert_{\mathcal{L}(B^{q}_{\Omega,\alpha,m} , A^{q,l}_{\Omega,\alpha,m})} \lesssim_{n,q,l} \vert x \vert^{-n}.
$$
\end{lemma}
\begin{proof}
Let $ f \in B^{q}_{\Omega,\alpha,m}$ then by definition
$$
K^{\epsilon}_{l}(x)(f)(r,t) = f(r) \partial^{l}_{t} P_{t+\epsilon}(x).
$$
Clearly this is linear. Taking the norm in $A^{q,l}_{\Omega,\alpha,m}$ gives us for $ x \neq 0$
\begin{equation*}
\begin{split}
&\vert K^{\epsilon}_{l}(x)(f) \vert_{A^{q,l}_{\Omega,\alpha,m}} 
\\
&= \left( \int_{\Omega} \alpha(r) \vert f(r) \vert^{q} \left( \int_{0}^{\infty} t^{2l-1} \vert \partial^{l}_{t} P_{t+\epsilon}(x) \vert^{2} dt \right)^{q/2} dr \right)^{1/q}
\\
& = \vert f\vert_{B^{q}_{\Omega,\alpha,m}} \left(\int_{0}^{\infty} t^{2l-1} \vert \partial^{l}_{t} P_{t+\epsilon}(x) \vert^{2} dt \right)^{1/2}
\\
& \lesssim_{n,l} \vert f \vert_{B^{q}_{\Omega,\alpha,m}} \vert x\vert^{-n}.
\end{split}
\end{equation*}
Where we have used the fact that for any $ l \in \mathbb{N}$ and $ x \neq 0 $ one has
\begin{equation}\label{ZZ}
\int_{0}^{\infty} t^{2l-1} \vert \partial^{l}_{t} P_{t+\epsilon}(x) \vert^{2} dt \lesssim_{n,l} \vert x \vert^{-2n}.
\end{equation}
Indeed, to see why this is true, notice that when differentiating the Poisson kernel $l$ times we have, when $ l$ is even
$$
\partial^{l}_{t} P_{t} (x) = c_{1} \frac{t}{(t^{2} + \vert x \vert^{2})^{\frac{n+1}{2} +\frac{l}{2}}} + c_{2} \frac{t^{3}}{(t^{2} + \vert x \vert^{2})^{\frac{n+1}{2}+\frac{l}{2}+1}}
+..+ c_{l/2} \frac{t^{l+1}}{(t^{2} + \vert x \vert^{2})^{\frac{n+1}{2}+l}}.
$$
And when $l$ is odd we get 
$$
\partial^{l}_{t} P_{t} (x) = c_{1}\frac{1}{(t^{2} + \vert x \vert^{2} )^{\frac{n+1}{2}+ \frac{l-1}{2}}} + c_{2}\frac{t^{2}}{(t^{2} + \vert x \vert^{2} )^{\frac{n+1}{2}+ \frac{l-1}{2}+1}}+...+c_{\frac{l-1}{2}}\frac{t^{l+1}}{(t^{2} + \vert x \vert^{2} )^{\frac{n+1}{2}+ l}}.
$$
In both cases, we can see the pointwise bound
$$
\vert \partial^{l}_{t} P_{t}(x) \vert \lesssim_{n,l} \frac{1}{(t^{2}+ \vert x \vert^{2} )^{\frac{n+1}{2} + \frac{l-1}{2}}}.
$$
From this pointwise bound we obtain 
\begin{equation*}
\begin{split}
& \int_{0}^{\infty} t^{2l-1} \vert \partial^{l}_{t} P_{t+\epsilon}(x) \vert^{2} dt 
\\
&
\leq \int_{0}^{\infty} t^{2l-1} \vert \partial^{l}_{t} P_{t}(x) \vert^{2} dt
\\
& \lesssim_{n,l} \int_{0}^{\infty} t^{2l-1} \frac{1}{( t^{2}+ \vert x \vert^{2})^{n+1+(l-1)}} dt
\\
& = \int_{0}^{\infty} \frac{t}{(t^{2}+ \vert x \vert^{2})^{n+1}} \frac{t^{2(l-1)}}{t^{2} + \vert x \vert^{2})^{l-1}} dt
\\
& \leq \int_{0}^{\infty} \frac{t}{(t^{2}+ \vert x \vert^{2})^{n+1}} 
\\
& \lesssim_{n} \vert x \vert^{-2n}.
\end{split}
\end{equation*}
This proves \eqref{ZZ}. If $ x = 0 $ then still we have $K^{\epsilon}_{l}(0) \in \mathcal{L}(B^{q}_{\Omega,\alpha,m} , A^{q,l}_{\Omega,\alpha,m})$ since
$$
\int_{0}^{\infty} t^{2l-1} \vert \partial^{l}_{t} P_{t+\epsilon}(0)\vert^{2} dt \lesssim_{n,l} \epsilon^{-2n}.
$$
This proves the lemma
\end{proof}
Next, we define operators between $L^{p}(B^{q}_{\Omega,\alpha,m})$ and $L^{p}(A^{q,l}_{\Omega,\alpha,m})$.
\begin{defn}
Fix $ l \in \mathbb{N}$. Then for $ f \in L^{p}(B^{q}_{\Omega,\alpha,m})$ we define $ T^{\epsilon}_{l}(f)$ as follows
$$
T_{l}(f)(x) = \int_{\mathbb{R}^{n}} K^{\epsilon}_{l}(x-y).f(y) dy.
$$
Where the integral is understood in the Bochner sense, that is for each fixed $ x $ and $f \in L^{p}(B^{q}_{\Omega,\alpha,m})$, we have that $ T^{\epsilon}_{l}(f)(x)$ is an element of $A^{q,l}_{\Omega,\alpha,m}$. Indeed, for each $y$ we have that $f(y) \in B^{q}_{\Omega,\alpha,m}$ by definition. Then $K^{\epsilon}_{l}(x-y)$ acts on $ f(y)$ by pointwise multiplication 
$$
(K^{\epsilon}_{l}(x-y).f(y))(t,r) = K^{\epsilon}_{l}(x-y)(f(y))(r,t) = f(y)(r) \partial^{l}_{t} P_{t+\epsilon}(x-y).
$$
\end{defn}
With this setup, we are ready to state the first part of the Littlewood-Paley theorem for these Banach valued functions

\begin{lemma}\label{K}
Let $ p, q \in (1,\infty)$ and $ B^{q}_{\Omega,m,\alpha}, A^{q,l}_{\Omega,\alpha,m}$ be as above, then for any $ f \in L^{p}(B^{q}_{\Omega, m , \alpha}) $ we have the following 
$$
\int_{\mathbb{R}^{n}} \vert T^{\epsilon}_{l}(f)(x) \vert_{A^{q,l}_{\Omega,\alpha,m}} dx \lesssim_{n,p,q,l}
\int_{\mathbb{R}^{n}} \vert f \vert^{p}_{B^{q}_{\Omega,m,\alpha}}.
$$

\end{lemma}
\begin{proof}
(To simplify notation we will omit $ \Omega, \alpha , m$ when writing the spaces $ B^{q}_{\Omega,m,\alpha}$ and $ A^{q,l}_{\Omega,m,\alpha}$). We aim to use Lemma \ref{Singular}. To do that, we need first to show that the kernel $K^{\epsilon}_{l}$ satisfy the estimates stated in Lemma \ref{Singular}. To that end, note that
from Lemma \ref{PKK} we have
$$
K^{\epsilon}_{l}: \mathbb{R}^{n} \setminus \{0 \} \to \mathcal{L}(B^{q} ,A^{q,l})
$$
with norm given by 
$$
\vert K^{\epsilon}_{l} \vert \lesssim_{n,l} \vert x \vert^{-n}.
$$
The calculations done in the proof of Lemma \ref{PKK} gives us a bound for the derivative, that is for any $ i \in \{ 1,....,n\}$ we have that 
$$
\vert \partial_{i} K^{\epsilon}_{l}(x) \vert_{\mathcal{L}(B^{q},A^{q})} \lesssim  \vert x \vert^{-n-1}.
$$
This shows that the kernel $K^{\epsilon}_{l}$ does satisfy the assumptions made in Lemma \ref{Singular}. Next, we need to show that the operator $T^{\epsilon}_{l}$ is bounded for at least one exponent. Notice that when $ q = p $ we have 
$$
\int_{\mathbb{R}^{n}} \vert T^{\epsilon}_{l}(f) \vert^{q}_{A^{q,l}} \lesssim_{n,q,l} \int_{\mathbb{R}^{n}} \vert f \vert^{q}_{B^{q}}.
$$
To see why this is true, start with $ f \in C^{\infty}_{0}(\mathbb{R}^{n} \times \Omega)$ then by definition
\begin{equation*}
\begin{split}
&\int_{\mathbb{R}^{n}} \vert T^{\epsilon}_{l}(f) \vert^{q}_{A^{q,l}}
\\
& = \int_{\mathbb{R}^{n}} \int_{\Omega} \alpha(r) \big[\int_{0}^{\infty} t^{2l-1} \vert \partial^{l}_{t} P_{t+\epsilon} * f(x,r) \vert^{2} dt \big]^{q/2} dr dx
\\
& \leq \int_{\mathbb{R}^{n}} \int_{\Omega} \alpha(r) \big[\int_{0}^{\infty} t^{2l-1} \vert \partial^{l}_{t} P_{t} * f(x,r) \vert^{2} dt \big]^{q/2} dr dx
\\
&
=\int_{\Omega}\alpha(r) \int_{\mathbb{R}^{n}}  \big[\int_{0}^{\infty} t^{2l-1} \vert \partial^{l}_{t} P_{t} * f(x,r) \vert^{2} dt \big]^{q/2} dx dr,
\end{split}
\end{equation*}
and by Lemma \ref{LP} 
\begin{equation*}
\begin{split}
& \approx_{n,q,l} \int_{\Omega} \alpha(r) \int_{\mathbb{R}^{n}} \vert f(x,r) \vert^{q} dx dr
\\ 
& = \int_{\mathbb{R}^{n}} \vert f \vert^{q}_{B^{q}}.
\end{split}
\end{equation*}
Then in light of Lemma \ref{density} we get the same for general $ f \in L^{p}(B^{q}) $. Now we conclude by Lemma \ref{Singular} the following inequality for $ p \in (1,\infty)$
$$
\int_{\mathbb{R}^{n}} \vert T^{\epsilon}_{l}(f) \vert^{p}_{A^{q,l}} \lesssim_{n,p,q,l} \int_{\mathbb{R}^{n}} \vert f \vert^{p}_{B^{q}}.
$$
Which is what we wanted to prove
\end{proof}

Now we can prove Lemma \ref{Kformixed}

\begin{proof}[proof of Lemma \ref{Kformixed}]
Take any $ f \in L^{p,q}_{\alpha}(\mathbb{R}^{n} \times \Omega)$. Then notice that $ f $ can be thought of as element of $L^{p}(B^{q}_{\Omega,\alpha,m})$, via the definition $ f(x)= [f(x,r)]$. Hence, by Lemma \ref{K} we get for any $ \epsilon > 0 $ the following 
$$
\int_{\mathbb{R}^{n}} \vert T_{l}^{\epsilon}(f)(x) \vert^{p}_{A^{q,l}_{\Omega,\alpha,m}} dx \lesssim_{n,p,q,l} \int_{\mathbb{R}^{n}} \vert f(x) \vert^{p}_{B^{q}_{\Omega,\alpha,m}} dx,
$$
which translates to 
\begin{equation*}
\begin{split}
&\int_{\mathbb{R}^{n}} \left( \int_{\Omega} \left( \int_{0}^{\infty} t^{2l-1} \vert \partial^{l}_{t} P_{t+\epsilon} * f(x,r) \vert^{2} dt\right)^{q/2} \alpha(r) dr \right)^{p/q} dx
\\
& 
\lesssim_{n,p,q,l} \int_{\mathbb{R}^{n}} \left( \int_{\Omega} \vert f(x,r) \vert^{q} \alpha(r) dr \right)^{p/q} dx,
\end{split}
\end{equation*}
for any $ \epsilon > 0$. Therefore, letting $ \epsilon \to 0$ gives us 
\begin{equation}\label{SSSS}
\begin{split}
&\int_{\mathbb{R}^{n}} \left( \int_{\Omega} \left( \int_{0}^{\infty} t^{2l-1} \vert \partial^{l}_{t} P_{t} * f(x,r) \vert^{2} dt\right)^{q/2} \alpha(r) dr \right)^{p/q} dx
\\
& 
\lesssim_{n,p,q,l} \int_{\mathbb{R}^{n}} \left( \int_{\Omega} \vert f(x,r) \vert^{q} \alpha(r) dr \right)^{p/q} dx.
\end{split}
\end{equation}
This proves the first direction. To get the other direction, notice that when $ p = 2 = q $ we have for any $ f \in L^{2,2}_{\alpha} ( \mathbb{R}^{n} \times \Omega)$ the following equality 
\begin{equation}\label{Sargu}
\begin{split}
&\int_{\mathbb{R}^{n}} \int_{\Omega} \left( \int_{0}^{\infty} t^{2l-1} \vert \partial^{l}_{t} P_{t} * f(x,r) \vert^{2} dt\right) \alpha(r) dr  dx
\\
& 
= C(n,l) \int_{\mathbb{R}^{n}}  \int_{\Omega} \vert f(x,r) \vert^{2} \alpha(r) dr dx.
\end{split}
\end{equation}
Indeed, this is a consequence of Fubini theorem and the classical Littlewood-Paley
\begin{equation}
\begin{split}
&\int_{\mathbb{R}^{n}} \left( \int_{\Omega} \left( \int_{0}^{\infty} t^{2l-1} \vert \partial^{l}_{t} P_{t} * f(x,r) \vert^{2} dt\right)^{q/2} \alpha(r) dr \right)^{p/q} dx
\\
& = \int_{\Omega} \alpha(r) \int_{0}^{\infty} t^{2l-1} \int_{\mathbb{R}^{n}} \vert \partial^{l}_{t} P_{t} * f(x,r) \vert^{2} dx dt dr 
\\
& = \int_{\Omega} \alpha(r) \int_{\mathbb{R}^{n}} \vert \xi \vert^{2l}  \vert \mathcal{F}_{x}(f) (\xi,r) \vert^{2} \int_{0}^{\infty} t^{2l-1} e^{- 4 \pi \vert \xi \vert t } dt d \xi dr
\\ 
& = \int_{\Omega} \alpha(r) \int_{\mathbb{R}^{n}} \vert \xi \vert^{2l} \vert \mathcal{F}_{x}(f) (\xi,r) \vert^{2} C \vert \xi \vert^{-2l} d \xi dr
\\
& = C \int_{\Omega} \alpha(r) \int_{\mathbb{R}^{n}} \vert f(x,r) \vert^{2} dx dr.
\end{split}
\end{equation}
Next, note that $L^{2,2}_{\alpha}(\mathbb{R}^{n} \times \Omega)$ is a Hilbert space with inner product given by
$$
\langle f,g \rangle_{L^{2,2}_{\alpha}} = \int_{\mathbb{R}^{n}} \int_{\Omega} f(x,r) g(x,r) \alpha(r) dr dx.
$$
Further, consider the space $L^{2,2}_{\alpha,l}$ defined by 
$$
L^{2,2}_{\alpha,l} = \{ F \in L_{loc}^{1}(\mathbb{R}^{n} \times \Omega \times ( 0 , \infty) ): \int_{\mathbb{R}^{n}} \int_{\Omega} \int_{0}^{\infty} t^{2l-1} \vert F(x,r,t) \vert^{2} dt \alpha(r) dr dx < \infty\}.
$$
Notice that this is also a Hilbert space under the inner product
$$
\langle F,G \rangle_{L^{2,2}_{\alpha,l}} = \int_{\mathbb{R}^{n}} \int_{\Omega} \int_{0}^{\infty} t^{2l-1} F(x,r,t) G(x,r,t) dt \alpha(r)dr dx.
$$
Therefore, using \eqref{Sargu} we get for $ f,g \in L^{2,2}_{\alpha}(\mathbb{R}^{n} \times \Omega)$ 
\begin{equation}\label{InnerP}
\begin{split}
& \int_{\mathbb{R}^{n}} \int_{\Omega} f(x,r) g(x,r) \alpha(r) dr dx
\\
& = \langle f,g \rangle_{L^{2,2}_{\alpha}} 
\\
& = c \left( \Vert f \Vert^{2}_{L_{\alpha}^{2,2}} + \Vert g \Vert^{2}_{L^{2,2}_{\alpha}} - \Vert g+f \Vert^{2}_{L^{2,2}_{\alpha}} \right)
\\
& = c \left( \Vert Tf \Vert_{L^{2,2}_{\alpha,l}} + \Vert g \Vert_{L^{2,2}_{\alpha,l}} - \Vert T(f+g) \Vert_{L^{2,2}_{\alpha,l}}\right)
\\
& = \langle Tf, Tg \rangle_{L^{2,2}_{\alpha,l}},
\end{split}
\end{equation}
where 
$$
T(f)(x,r,t) = \partial^{l}_{t} P_{t} *f(x,r),
$$
the convolution is in the $ x $ variable. Fix $ f \in L^{2,2}_{\alpha}(\mathbb{R}^{n} \times \Omega) \cap L^{p,q}_{\alpha}(\mathbb{R}^{n} \times \Omega)$. Then for any $ g \in L^{2,2}_{\alpha}(\mathbb{R}^{n} \times \Omega)\cap L^{p',q'}_{\alpha}(\mathbb{R}^{n} \times \Omega) $ we get from using \eqref{InnerP} 
\begin{equation*}
\begin{split}
& \int_{\mathbb{R}^{n}} \int_{\Omega} f(x,r) g(x,r) \alpha(r) dr dx
\\
& = \int_{\mathbb{R}^{n}} \int_{\Omega} \alpha(r) \int_{0}^{\infty} t^{2l-1} (\partial^{l}_{t}P_{t} * f(x,r) ).( \partial^{l}_{t}P_{t}*g(x,r)) dt dr dx.
\end{split}
\end{equation*}
Applying H\"older inequality three times we obtain 
\begin{equation*}
\begin{split}
&\int_{\mathbb{R}^{n}} \int_{\Omega} f(x,r) g(x,r) \alpha(r) dr dx
\\
& \leq \left( \int_{\mathbb{R}^{n}} \left( \int_{\Omega} \left( \int_{0}^{\infty} t^{2l-1} \vert \partial^{l}_{t} P_{t} * f(x,r) \vert^{2} dt\right)^{q/2} \alpha(r) dr \right)^{p/q} dx\right)^{1/p}
\\
&. \left( \int_{\mathbb{R}^{n}} \left( \int_{\Omega} \left( \int_{0}^{\infty} t^{2l-1} \vert \partial^{l}_{t} P_{t} * g(x,r) \vert^{2} dt\right)^{q'/2} \alpha(r) dr \right)^{p'/q'} dx\right)^{1/q'}.
\end{split}
\end{equation*}
Using \eqref{SSSS} we obtain
\begin{equation}\label{ABA}
\begin{split}
& \int_{\mathbb{R}^{n}} \int_{\Omega} f(x,r) g(x,r) \alpha(r) dr dx
\\
& \lesssim_{n,p,q,l} \Vert g \Vert_{L^{p',q'}_{\alpha}} .\left( \int_{\mathbb{R}^{n}} \left( \int_{\Omega} \left( \int_{0}^{\infty} t^{2l-1} \vert \partial^{l}_{t} P_{t} * f(x,r) \vert^{2} dt\right)^{q/2} \alpha(r) dr \right)^{p/q} dx\right)^{1/p}.
\end{split}
\end{equation}
Taking the supremum over all $g \in L^{2,2}_{\alpha}(\mathbb{R}^{n} \times \Omega) \cap L^{p',q'}_{\alpha}(\mathbb{R}^{n} \times \Omega)$ with norm $1 $ in $ L^{p',q'}_{\alpha}(\mathbb{R}^{n} \times \Omega)$ gives us the result for when $ f \in L^{2,2}_{\alpha}(\mathbb{R}^{n} \times \Omega) \cap L^{p,q}_{\alpha}(\mathbb{R}^{n} \times \Omega)$. Indeed, by duality, Lemma \ref{Lastdual}, we have that 
$$
\Vert f \Vert_{L^{2,2}_{\alpha}} = \sup_{g \in L^{p',q'}_{\alpha}, \Vert g \Vert_{L^{q',p'}_{\alpha} \leq 1}} \int_{\mathbb{R}^{n}} \int_{\Omega} f(x,r) g(x,r) \alpha(r) dr dx.
$$
And since we have from Lemma \ref{density} that $C^{\infty}_{0}(\mathbb{R}^{n} \times \Omega)$ is dense in $L^{p'}(B^{q'}_{\Omega,\alpha,m})$ then we get that $C^{\infty}_{0}(\mathbb{R}^{n} \times \Omega)$ is also dense in $L^{p',q'}_{\alpha}(\mathbb{R}^{n} \times \Omega)$ and hence $ L^{2,2}_{\alpha}(\mathbb{R}^{n} \times \Omega)$ is dense in $L_{\alpha}^{p',q'}(\mathbb{R}^{n} \times \Omega)$. Thus, it suffices to take the supremum over $g \in L^{2,2}_{\alpha} \cap L^{p',q'}_{\alpha}$. Now we can conclude from \eqref{ABA} that 
\begin{equation*}
\begin{split}
&\int_{\mathbb{R}^{n}} \left( \int_{\Omega} \vert f(x,r) \vert^{q} \alpha(r) dr \right)^{p/q} dx
\\
&
\lesssim_{n,p,q,l} \int_{\mathbb{R}^{n}} \left( \int_{\Omega} \left( \int_{0}^{\infty} t^{2l-1} \vert \partial^{l}_{t} P_{t} * f(x,r) \vert^{2} dt\right)^{q/2} \alpha(r) dr \right)^{p/q} dx.
\end{split}
\end{equation*}
This holds for any $ f \in L^{2,2}_{\alpha}(\mathbb{R}^{n} \times \Omega) \cap L^{p,q}_{\alpha}(\mathbb{R}^{n} \times \Omega)
$, the general result follows by density.

\end{proof}

\bibliographystyle{abbrv}
\bibliography{bib}

\end{document}